\documentclass{amsart}

\title{New jump operators on equivalence relations}

\author{John D. Clemens}
\address{1910 University Drive, Boise, ID 83725}
\email{johnclemens@boisestate.edu}
\author{Samuel Coskey}
\address{1910 University Drive, Boise, ID 83725}
\email{scoskey@boisestate.edu}

\usepackage{amssymb}
\usepackage[hidelinks]{hyperref}

\renewcommand{\setminus}{\smallsetminus}
\newcommand{\N}{{\mathbb N}}
\newcommand{\Z}{{\mathbb Z}}
\newcommand{\Q}{{\mathbb Q}}
\newcommand{\R}{{\mathbb R}}
\newcommand{\F}{{\mathbb F}}
\newcommand{\PP}{{\mathbb P}}
\DeclareMathOperator{\dom}{dom}

\DeclareMathOperator{\Aut}{Aut}
\DeclareMathOperator{\pot}{pot}
\DeclareMathOperator{\ran}{ran}
\DeclareMathOperator{\SOT}{SOT}

\def\underTilde#1{{\baselineskip=0pt\vtop{\hbox{$#1$}\hbox{$\sim$}}}{}}

\def\bSigma{\underTilde{\Sigma}}
\def\bPi{\underTilde{\Pi}}
\def\bDelta{\underTilde{\Delta}}

\newtheorem{thm}{Theorem}[section]
\newtheorem{lem}[thm]{Lemma}
\newtheorem{prop}[thm]{Proposition}
\newtheorem{cor}[thm]{Corollary}
\newtheorem{introthm}{Theorem}
\newtheorem*{thm*}{Theorem}
\theoremstyle{definition}
\newtheorem{defn}[thm]{Definition}
\newtheorem{question}{Question}
\newtheorem{remark}[thm]{Remark}

\usepackage{etoolbox}
\makeatletter\pretocmd{\@seccntformat}{\S}{}{}
  \pretocmd{\@subseccntformat}{\S}{}{}\makeatother

\usepackage{xifthen}
\newcommand*{\J}[3][]{J^{[#2]}_{#3}\ifthenelse{\isempty{#1}}{}{\left(#1\right)}}

\begin{document}

\begin{abstract}
  We introduce a new family of jump operators on Borel equivalence relations; specifically, for each countable group $\Gamma$ we introduce the $\Gamma$-jump. We study the elementary properties of the $\Gamma$-jumps and compare them with other previously studied jump operators. One of our main results is to establish that for many groups $\Gamma$, the $\Gamma$-jump is \emph{proper} in the sense that for any Borel equivalence relation $E$ the $\Gamma$-jump of $E$ is strictly higher than $E$ in the Borel reducibility hierarchy. On the other hand there are examples of groups $\Gamma$ for which the $\Gamma$-jump is not proper. To establish properness, we produce an analysis of Borel equivalence relations induced by continuous actions of the automorphism group of what we denote the full $\Gamma$-tree, and relate these to iterates of the $\Gamma$-jump. We also produce several new examples of equivalence relations that arise from applying the $\Gamma$-jump to classically studied equivalence relations and derive generic ergodicity results related to these. We apply our results to show that the complexity of the isomorphism problem for countable scattered linear orders properly increases with the rank.
\end{abstract}

\maketitle

\section{Introduction}

The backdrop for our study is the Borel complexity theory of equivalence relations. Recall that if $E,F$ are equivalence relations on standard Borel spaces $X,Y$ then $E$ is \emph{Borel reducible} to $F$, written $E\leq_B F$, if there exists a Borel function $f\colon X\to Y$ such that
\[x\mathrel{E}x'\iff f(x)\mathrel{F}f(x')\text{.}
\]
We say $f$ is a \emph{homomorphism} if it satisfies the left-to-right implication.
We write $E\sim_B F$ if both $E\leq_B F$ and $F\leq_B E$, and we write $E<_B F$ if both $E\leq_B F$ and $E\not\sim_B F$.

The notion of Borel reducibility gives rise to a preorder structure on equivalence relations. As with other complexity hierarchies, it is natural to study operations such as jumps.

\begin{defn}
  \label{def:jump}
  We say that a mapping $E \mapsto J(E)$ on Borel equivalence relations is a \emph{proper jump operator} if it satisfies the following properties for Borel equivalence relations $E,F$:
  \begin{itemize}
    \item Monotonicity: $E \leq_B F$ implies $J(E) \leq_B J(F)$;
    \item Properness: $E <_B J(E)$ whenever $E$ has at least two equivalence classes.
  \end{itemize}
\end{defn}

Note that the terms \emph{jump} or \emph{jump operator} may be used for a monotone mapping with $E \leq_B J(E)$ in a context where strict properness is not relevant or has not been established. We may also apply a jump operator to analytic equivalence relations; in this case we do not expect or require properness. Indeed, for all of the jump operators discussed below, one can find analytic equivalence relations which are fixed points for the mapping up to Borel bireducibility.

One may also ask for some definability condition on a jump operator. While we do not require any particular such conditions, it is the case that all the jump operators discussed below are uniformly definable in the sense that if $R \subseteq Y \times X^2$ is a Borel set so that each section $R_y$ is an equivalence relation on $X$, then the set $\tilde{R} \subseteq Y \times \tilde{X}^2$ given by $\tilde{R}(y, \tilde{x}_1,\tilde{x}_2)$ iff $\tilde{x}_1 \mathrel{J(E_y)} \tilde{x}_2$ is also Borel, where $\tilde{X}$ denotes the domain of $J(E)$.

Several jump operators have been studied extensively, including the Friedman--Stanley jump \cite{friedman-stanley} and the Louveau jump \cite{louveau}, which we discuss below. There is also a jump operator on quasi-orders introduced by Rosendal \cite{rosendal}; see also subsequent work by Camerlo et al \cite{marcone}.

Here we introduce a new class of jump operators which are associated with countable groups.

\begin{defn}
  Let $E$ be an equivalence relation on $X$, and let $\Gamma$ be a countable group. The \emph{$\Gamma$-jump} of $E$ is the equivalence relation $E^{[\Gamma]}$ defined on $X^\Gamma$ by
  \[x\mathrel{E}^{[\Gamma]}y \iff (\exists\gamma\in\Gamma)\; (\forall\alpha\in\Gamma)\; x(\gamma^{-1}\alpha)\mathrel{E}y(\alpha)\text{.}
  \]
\end{defn}

We will use the term \emph{Bernoulli jump} as a collective name for any member of the family of $\Gamma$-jumps. Indeed, note that if $E=\Delta(2)$ then $E^{[\Gamma]}$ is the orbit equivalence relation induced by the classical Bernoulli shift action of $\Gamma$, and if $E=\Delta(X)$ for a Polish space $X$ then $E^{[\Gamma]}$ is the orbit equivalence relation induced by the ``generalized'' Bernoulli action of $\Gamma$.

We reserve the notation $E^\Gamma$ for the product of countably many copies of $E$ with index set $\Gamma$. Thus $E^{\Gamma}$ is Borel isomorphic to $E^{\omega}$, and $E^{[\Gamma]}$ is an equivalence relation of countable index over $E^\Gamma$. Indeed, letting $\Gamma$ act on $X^{\Gamma}$ by the left shift $\gamma \cdot x (\alpha) = x (\gamma^{-1}\alpha)$, we have that $x,y$ are $E^{[\Gamma]}$-equivalent iff there is $\gamma\in\Gamma$ with $\gamma \cdot x \mathrel{E^{\Gamma}} y$.

It is clear that the $\Gamma$-jump operator is monotone for any $\Gamma$. We will be concerned with whether and when the $\Gamma$-jump is proper. Before addressing this question, we recall the situation with the Friedman--Stanley and Louveau jumps.

\begin{defn}
  Let $E$ be a Borel equivalence relation on $X$. The \emph{Friedman--Stanley jump} of $E$ is the equivalence relation $E^+$ defined on $X^{\omega}$ by
  \[x\mathrel{E^+}y \iff \{ [x(n)]_E : n \in \omega\} = \{[y(n)]_E : n \in \omega\}\text{.}
  \]
\end{defn}

\begin{thm*}[Friedman--Stanley, \cite{friedman-stanley}]
  The mapping $E \mapsto E^{+}$ is a proper jump operator.
\end{thm*}

\begin{defn}
  Let $E$ be a Borel equivalence relation on $X$ and let $\mathcal{F}$ be a free filter on $\omega$. The \emph{Louveau jump} of $E$ with respect to $\mathcal F$ is the equivalence relation $E^{\mathcal{F}}$ defined on $X^{\omega}$ by
  \[ x \mathrel{E^{\mathcal{F}}}y \iff \{n\in\omega:x(n)\mathrel{E}y(n)\} \in \mathcal{F}\text{.}
  \]
\end{defn}

\begin{thm*}[Louveau, \cite{louveau}]
  For any free filter $\mathcal{F}$, the mapping $E\mapsto E^{\mathcal{F}}$ is a proper jump operator.
\end{thm*}

The original proof of the Friedman--Stanley result used Friedman's theorem on the non-existence of Borel diagonalizers (recall a \emph{Borel diagonalizer} for $E$ is a homomorphism $\varphi$ from $E^{+}$ to $E$ so that $\varphi(x) \notin \{ [x_n]_E : n \in \omega\}$). However, both the Friedman--Stanley result and the Louveau result can be proved using the concept of potential complexity of equivalence relations, which we briefly introduce. First we will say that a \emph{Borel class} is a pointclass $\Gamma$ consisting of Borel sets and closed under continuous preimages. For example, $\bPi^0_{\alpha}$ is a Borel class for any $\alpha<\omega_1$.

\begin{defn}
  Let $E$ be an equivalence relation on the Polish space $(X,\tau)$, and let $\Gamma$ be a Borel class. We say $E$ is \emph{potentially $\Gamma$}, written $E \in \pot(\Gamma)$, if there is a topology $\sigma$ on $X$ such that $\sigma,\tau$ have the same Borel sets and such that $E$ is in $\Gamma$ with respect to $\sigma$.
\end{defn}

We remark that $E$ is potentially $\Gamma$ if and only if $E$ is \emph{essentially} $\Gamma$, i.e., there is some $E'$ in $\Gamma$ with $E \leq_B E'$.

\begin{defn}
  We say that a family $\mathfrak{F}$ of Borel equivalence relations has \emph{cofinal potential complexity} if for for every Borel class $\Gamma$ there is $E \in \mathfrak{F}$ such that $E \notin \pot(\Gamma)$.
\end{defn}

Louveau established that if $E$ is an equivalence relation with at least two classes, then the family of iterates of the Louveau jump of $E$ with respect to a free filter has cofinal potential complexity. Since being potentially $\Gamma$ is equivalent to being essentially $\Gamma$, a family of cofinal potential complexity cannot have a maximum element with respect to $\leq_B$. Hence it follows that the Louveau jump with respect to a free filter is a proper jump operator.

Meanwhile it is known that the family of Borel equivalence relations induced by actions of $S_\infty$ has cofinal potential complexity. Moreover Hjorth--Kechris--Louveau \cite{hjorth-kechris-louveau} established that if $E$ is an equivalence relation with at least two classes, then any Borel equivalence relation induced by an action of $S_\infty$ is Borel reducible to some iterate of the Friedman--Stanley jump of $E$. Thus they arrived at a ``potential complexity'' proof that the Friedman--Stanley jump is a proper jump operator.

Returning to the Bernoulli jumps, we will establish the following.

\begin{introthm}
  \label{introthm:proper}
  Let $\Gamma$ be a countable group so that $\Z$ or $\Z_p^{<\omega}$ for $p$ prime is a quotient of a subgroup of $\Gamma$. Then the mapping $E\mapsto E^{[\Gamma]}$ is a proper jump operator.
\end{introthm}

In the proof, we will use a result of Solecki \cite{solecki} which implies that if $\Gamma$ is one of the groups $\Z$ or $\Z_p^{<\omega}$ for $p$ prime, then the family of Borel equivalence relations induced by actions of the group $\Gamma^{\omega}$ has cofinal potential complexity. Our main work will be to show that such equivalence relations are Borel reducible to iterates of the $\Gamma$-jump. Before stating this result, we provide the following notation for the iterates of the $\Gamma$-jump.

\begin{defn}
  For an equivalence relation $E$ and a countable group $\Gamma$ we define the iterates $\J[E]{\Gamma}{\alpha}$ of the $\Gamma$-jump recursively by:
  \begin{align*}
  \J[E]{\Gamma}{0} &=E \\
    \J[E]{\Gamma}{\alpha+1}&=(\J[E]{\Gamma}{\alpha})^{[\Gamma]}\\
    \J[E]{\Gamma}{\lambda}&=
    \left(\bigoplus_{\alpha<\lambda}\J[E]{\Gamma}{\alpha}\right)^{[\Gamma]}\text{for $\lambda$ a limit} .
  \end{align*}
  We use $\J{\Gamma}{\alpha}$ to denote $\J[\Delta(2)]{\Gamma}{\alpha}$ and $Z_{\alpha}$ to denote $\J[\Delta(2)]{\Z}{\alpha}$.
\end{defn}

The tower of $Z_\alpha$ is of particular interest and will figure in an application to the classification of scattered linear orders. We note that $\Delta(2)$ may be replaced by any nontrivial Polish space to produce equivalent iterates for $\alpha \geq 1$. Also, in the definition of $\J[E]{\Gamma}{\alpha}$, the third rule may be used for successor ordinals as well. 

Returning to the proof of Theorem 1, we will actually consider a more complicated group than $\Gamma^{\omega}$. For a countable group $\Gamma$, we will introduce the notion of a \emph{$\Gamma$-tree}, which is a tree in which the children of each node carry the structure of a subset of the group $\Gamma$. We will see that $\Gamma$-trees are closely tied to iterates of the $\Gamma$-jump in a way analogous to the way regular trees are tied to iterates of the Friedman--Stanley jump. Namely, the iterated $\Gamma$-jump $\J{\Gamma}{\alpha}$ will be Borel bireducible with the isomorphism relation on well-founded $\Gamma$-trees of rank $1+\alpha$.

In particular, we will introduce the \emph{full $\Gamma$-tree $T_{\Gamma}$}, which may be naturally identified with $\Gamma^{< \omega}$. Its automorphism group, $\Aut(T_{\Gamma})$, is a closed subgroup of $S_{\infty}$, and the group $\Gamma^{\omega}$ is a closed subgroup of $\Aut(T_{\Gamma})$. We will establish the following.

\begin{introthm}
  \label{introthm:reduce}
  Let $\Gamma$ be a countable group. Then for any Borel equivalence relation $E$ such that $E$ is induced by a Borel action of a closed subgroup of $\Aut(T_{\Gamma})$, there exists $\alpha$ such that $E$ is Borel reducible to $\J{\Gamma}{\alpha}$.
\end{introthm}

Since this result applies to the group $\Gamma^{\omega}$, when it is combined with the result of Solecki showing such groups have non-Borel orbit equivalence relations, it is sufficient to prove Theorem~1.

The Bernoulli jumps are not always proper jump operators. In particular, we will establish the following.

\begin{introthm}
  \label{introthm:improper}
  Let $\Gamma$ be a countable group with no infinite sequence of strictly descending subgroups. Then the mapping $E\mapsto E^{[\Gamma]}$ is not a proper jump operator.
\end{introthm}

To establish this result, we directly calculate that for such $\Gamma$,we have that $\J{\Gamma}{\omega+1}$ is Borel bireducible with $\J{\Gamma}{\omega}$.

Theorems~1 and~3 leave open the question of precisely when the $\Gamma$-jump is proper. Indeed, we shall see that there exist groups $\Gamma$ which do not meet the hypothesis of either result.

We will also study the structure of specific Borel equivalence relations with respect to the Bernoulli jumps and see that they provide new examples of equivalence relations whose complexity lies between $E_{\infty}$ and $F_2$. We first establish:

\begin{introthm}
  \label{introthm:generic}
  $E_0^{[\Gamma]}$ is generically $E_{\infty}^{\omega}$-ergodic.
\end{introthm}

This result has subsequently been strengthened by Allison and Panagiotopoulos \cite{Allison-Aristotelis} to show that $E_0^{[\Z]}$ is generically ergodic  with respect to any orbit equivalence relation of a TSI Polish group. Using this, we establish:

\begin{introthm}
  \label{introthm:intermediate}
  We have the following:
  \begin{itemize}
    \item $E_0^{\omega} <_B E_0^{[\Z]} <_B E_{\infty}^{[\Z]} <_B F_2$.
    \item $E_0^{\omega} <_B E_{\infty}^{\omega} <_B E_{\infty}^{[\Z]} <_B F_2$.
    \item $E_0^{[\Z]}$ and $E_{\infty}^{\omega}$ are $\leq_B$-incomparable.
  \end{itemize}
\end{introthm}

Shani \cite{shani} has produced further non-reducibility results about Bernoulli jumps of countable Borel equivalence relations. For instance, he has shown that $E_0^{[\Z]} <_B E_0^{[\Z^2]}$, and that $E_0^{[\Z]}$ and $E_0^{[\Z_2^{< \omega}]}$ are $\leq_B$-incomparable. He has also produced another example of an equivalence relation strictly between $E_{\infty}^{\omega}$ and $F_2$ which is $\leq_B$-incomparable with $E_0^{[\Z]}$ and $E_{\infty}^{[\Z]}$.

One application of the theory of Bernoulli jumps is to the classification of countable scattered linear orders. Recall that a linear order is \emph{scattered} if it has no subordering isomorphic to $(\Q,<)$. The class of scattered linear orders carries a derivative operation where one identifies points $x,y$ such that the interval $[x,y]$ is finite. This derivative operation may furthermore be used to define a rank function on the countable scattered orders with values in $\omega_1$, as we will discuss in Section~\ref{sec:scattered}. We will establish the following.

\begin{introthm}
  \label{introthm:scattered}
  The isomorphism equivalence relation on countable scattered linear orders of rank $1+\alpha$ is Borel bireducible with $\J{\Z}{\alpha}$.
\end{introthm}

This result, together with the fact that the $\Z$-jump is proper, implies that the complexity of the classification of countable scattered linear orders increases strictly with the rank.

This paper is organized as follows. In the next section we establish some of the basic properties of the $\Gamma$-jump. In Section~3 we compare the $\Gamma$-jump with the Friedman--Stanley jump. In Section~4 we introduce $\Gamma$-trees and relate iterates of the $\Gamma$-jump to isomorphism of well-founded $\Gamma$-trees. In Section~5 we study actions of $\Aut(T_{\Gamma})$ and prove Theorem~\ref{introthm:reduce}. In Section~6 we investigate the properness of the $\Gamma$-jump and in particular prove Theorem~\ref{introthm:proper} and Theorem~\ref{introthm:improper}. In Section~7 we revisit the proof from Section~5 and show that in some cases, we can achieve better lower bounds on the complexity of the iterated jumps. In Section~8 we investigate new equivalence relations arising from Bernoulli jumps, and compare them against well-known equivalence relations, establishing Theorems~\ref{introthm:generic} and \ref{introthm:intermediate}. In Section~9 we discuss the connection between the $\Z$-jump and scattered linear orders, and in particular prove Theorem~\ref{introthm:scattered}.

\textbf{Acknowledgement.} We would like to thank Shaun Allison, Ali Enayat, Aristoteles Panagiotopoulos, and Assaf Shani  for helpful discussions about the content of this article.
We would also like to thank the referee for several helpful suggestions, in particular for pointing out that our original complexity bounds in Proposition~\ref{prop:sharp-bounds} were not correct.

\section{Properties of the $\Gamma$-jump for a countable group $\Gamma$}
\label{sec:properties}

In this section we explore some of the basic properties of the $\Gamma$-jump. We begin by recording the following three properties, whose proofs are immediate.

\begin{prop}
  \label{prop:jumps-basic}
  For any countable group $\Gamma$ and equivalence relations $E$ and $F$ we have:
  \begin{enumerate}
    \item If $E$ is Borel (resp.\ analytic), then $E^{[\Gamma]}$ is Borel (resp.\ analytic).
    \item $E \leq_B E^{[\Gamma]}$.
    \item If $E \leq_B F$ then $E^{[\Gamma]} \leq_B F^{[\Gamma]}$.
  \end{enumerate}
\end{prop}

The next result strengthens Proposition~\ref{prop:jumps-basic}(b).

\begin{prop}
  \label{prop:Eomega}
  If $E$ is a Borel equivalence relation and $\Gamma$ is an infinite group, then $E^\omega\leq_B E^{[\Gamma]}$.
\end{prop}

\begin{proof}
  If $E$ has just finitely many classes, the left-hand side is smooth. Otherwise for any $a\in X$, we have that $E$ is Borel bireducible with $E\restriction X\setminus[a]$. Hence it is sufficient to define a reduction from $(E\restriction X\setminus[a])^\omega$ to $E^{[\Gamma]}$.
  
  Let $S\subset\Gamma$ be an infinite subset such that $\gamma\neq1$ implies $\gamma S\neq S$. (The set of such $S$ is of measure $1$.) Define a function $f\colon (X\setminus[a])^S\to X^\Gamma$ by
  \[f(x)(\alpha)=\begin{cases}
    x(\alpha),&\text{if $\alpha\in S$,} \\
    a,&\text{otherwise.}
  \end{cases}\]
  Clearly if $x\mathrel{E^S}x'$ then $f(x)\mathrel{E^\Gamma}f(x')$ and therefore $f(x)\mathrel{E^{[\Gamma]}}f(x')$. For the reverse implication suppose $f(x)\mathrel{E^{[\Gamma]}}f(x')$, and let $\gamma\in\Gamma$ be such that $\gamma \cdot f(x)\mathrel{E^\Gamma}f(x')$. Then $f(x')(\gamma \alpha)\mathrel{E}f(x)(\alpha)$, so
  \[f(x)(\alpha)=a \iff f(x')(\gamma\alpha)=a .\]
  This means precisely that $\gamma S=S$. It follows that $\gamma=1$, and therefore that $x\mathrel{E^S}x'$.
\end{proof}

The $\Gamma$-jump may be thought of as a kind of wreath product. Recall that if $\Lambda,\Gamma$ are groups then the full support wreath product $\Lambda\wr\Gamma$ is defined as follows. Let $\Lambda^\Gamma$ be the full support group product of $\Gamma$ many copies of $\Lambda$, and let $\Gamma$ act on $\Lambda^\Gamma$ by the shift. Then $\Lambda\wr\Gamma$ is the semidirect product $\Gamma\rtimes\Lambda^\Gamma$ with respect to this action.

\begin{prop}
  \label{prop:wr}
  If $E_\Lambda$ is the orbit equivalence relation of some (possibly uncountable) group $\Lambda$ acting on $X$, and $\Gamma$ is any countable group, then $(E_\Lambda)^{[\Gamma]}$ is the orbit equivalence relation of $\Lambda\wr\Gamma$ acting on $X^\Gamma$.
\end{prop}

This result implies that the $\Gamma$-jump preserves many properties of orbit equivalence relations. For the next statement, recall that a \emph{cli group} is one which carries a complete, left-invariant metric.

\begin{cor}
  \label{cor:preservation}
  If $E$ is the orbit equivalence relation of a Polish group then $E^{[\Gamma]}$ is the orbit equivalence relation of a Polish group. The same statement holds if we replace ``Polish group'' with any of the following: solvable group, cli group, or closed subgroup of $S_\infty$.
\end{cor}

This follows for cli groups from the preservation of being a cli group under wreath products (Theorem~2.2.11 of \cite{gao}). Note, however, that being induced by a TSI group is not preserved (here a \emph{TSI group} is one which carries a two-sided invariant metric). Indeed, it is shown in \cite{Allison-Aristotelis} that $E_0^{[\Z]}$ can not be the orbit equivalence relation of a TSI group.

By contrast, the Louveau jump does not satisfy any of these preservation properties. For example, the Louveau jump of $\Delta(\R)$ with respect to the Fr\'echet filter is $E_1$, and it is well-known $E_1$ is not reducible to a Polish group action \cite{harrington-kechris-louveau}. The Friedman--Stanley jump does satisfy the above preservation property with respect to Polish groups and subgroups of $S_\infty$. However it does not satisfy the above preservation property with respect to solvable groups, TSI groups, or cli groups. For example, the Friedman--Stanley jump of $\Delta(\R)$ is $F_2$ and it is well-known $F_2$ is not induced by a cli group action \cite{kanovei}.

\begin{prop}
  \label{prop:subquo}
  Let $E$ be an equivalence relation on $X$. If $\Lambda$ is a subgroup or quotient of $\Gamma$, then $E^{[\Lambda]}\leq_BE^{[\Gamma]}$.
\end{prop}

\begin{proof}
  We adapt the argument from the case when $E$ is an equality relation (see Proposition~7.3.4 of \cite{gao}). If $\Lambda$ is a quotient of $\Gamma$, let $\pi\colon\Gamma\to\Lambda$ be a surjective homomorphism and define $f\colon X^\Lambda\to X^\Gamma$ by $f(x)=x\circ\pi$. Then whenever $\lambda=\pi(\gamma)$ we have
  \begin{align*}
    \lambda x\mathrel{E^\Lambda}x'
    &\iff (\forall\alpha\in\Lambda)\;x(\lambda^{-1}\alpha)\mathrel{E} x'(\alpha)\\
    &\iff (\forall\beta\in\Gamma)\;x(\pi(\gamma^{-1}\beta))\mathrel{E}x'(\pi(\beta))\\
    &\iff (\forall\beta\in\Gamma)\;f(x)(\gamma^{-1}\beta)\mathrel{E}f(x')(\beta)\\
    &\iff \gamma \cdot f(x)\mathrel{E^\Gamma}f(x')\text{.}
  \end{align*}
  It follows that $f$ is a reduction from $E^{[\Lambda]}$ to $E^{[\Gamma]}$.

  Next assume that $\Lambda\leq\Gamma$. Fix an element $a\in X$, and let $f\colon X^\Lambda\to X^\Gamma$ be defined by
  \[f(x)(\alpha)=\begin{cases}
  x(\alpha),&\alpha\in\Lambda\text{,}\\
  a,&\text{otherwise.}
  \end{cases}
  \]
  Clearly if $\lambda \cdot x\mathrel{E^\Lambda}x'$ then $\lambda \cdot f(x)\mathrel{E^\Gamma}f(x')$ as well, which means $x\mathrel{E^{[\Lambda]}}x'$ implies $f(x)\mathrel{E^{[\Gamma]}}f(x')$. For the reverse implication, suppose $f(x)\mathrel{E^{[\Gamma]}}f(x')$, and let $\gamma\in\Gamma$ be such that $\gamma \cdot f(x)\mathrel{E^\Gamma}f(x')$. If $\gamma=\lambda\in\Lambda$, it is clear that $\lambda \cdot x\mathrel{E^\Lambda}f(x')$ and so $x\mathrel{E^{[\Lambda]}}x'$. On the other hand if $\gamma\notin\Lambda$, then for all $\lambda\in\Lambda$, we have
  \[x'(\lambda)=f(x')(\lambda)\mathrel{E}\gamma \cdot f(x)(\lambda)
  =f(x)(\gamma^{-1}\lambda)=a .
  \]
  An identical calculation with $x,x'$ exchanged and $\gamma,\gamma^{-1}$ exchanged shows the same for $x$. Thus in this case $x\mathrel{E^\Lambda}x'$ and hence we have $x\mathrel{E^{[\Lambda]}}x'$, as desired.
\end{proof}

Next we can relate jumps for finite powers of a group $\Gamma$ to iterates of the jump for $\Gamma$.

\begin{lem}
  \label{lem:Zk}
  For any countable group $\Gamma$, we have $E^{[\Gamma^k]} \leq_B \J[E^{\Gamma^k}]{\Gamma}{k}$. In particular, $E^{[\Gamma^k]} \leq_B \J[E]{\Gamma}{k+1}$.
\end{lem}

\begin{proof}
  Beginning with the first statement, we show the case of $k=2$, with larger $k$ being similar. Given $x\in X^{\Gamma\times\Gamma}$ let $f_x$ be the function from $\Gamma\times\Gamma$ to $X^{\Gamma \times \Gamma}$ so that $f_x(\alpha,\beta)$ is a code for all of the $x(\gamma,\delta)$, viewed from the base point $(\alpha,\beta)$, i.e., $f_x(\alpha,\beta)(\gamma,\delta) = x(\alpha \gamma, \beta \delta)$.We now define a reduction $\varphi$ from $E^{[\Gamma^2]}$ to $((E^{\Gamma^2})^{[\Gamma]})^{[\Gamma]}$ by setting $\varphi(x)(\alpha)(\beta) = f_x(\alpha,\beta)$

  If $x$ is equivalent to $y$, witnessed by $(\alpha,\beta)$, then $\varphi(x)$ is equivalent to $\varphi(y)$ with the  witness $\alpha$ for the outer jump and $\beta$ for each coordinate of the inner jump. 
Conversely if $\varphi(x)$ is equivalent to $\varphi(y)$, then there exists $(\alpha,\beta)$ such that $f_x(1,1)$ is equivalent to $f_y(\alpha,\beta)$. It follows that $(\alpha,\beta)$ witnesses that $x$ is equivalent to $y$.

  The second statement follows from Proposition~\ref{prop:Eomega}.
\end{proof}

We can also absorb countable powers in certain $\Gamma$-jumps.

\begin{prop}
  \label{prop:Eomegajump}
  If $E$ is a Borel equivalence relation and $\Gamma$ and $\Delta$ are infinite groups, then $\left(E^{\omega}\right)^{[\Gamma]} \leq_B E^{[\Gamma \times \Delta]}$.
\end{prop}

\begin{proof}
  Arguing as in Proposition~\ref{prop:Eomega}, we may assume $E$ has infinitely many classes, and find a reduction from $\left((E\restriction X\setminus[a])^\omega\right)^{[\Gamma]}$ to $E^{[\Gamma \times \Delta]}$ for some $a \in X$.
Let $\Delta=\{\delta_n : n \in \omega\}$. For $x \in \left((X\setminus [a])^{\omega}\right)^{\Gamma}$, define $f(x)$ by
  \[ f(x)((\alpha, \delta_n))= \begin{cases}
x(\alpha)(n-1), & \text{if $n >1$,} \\
a, & \text{if $n=0$.}
\end{cases}\]
If $x \mathrel{\left(E^{\omega}\right)^{[\Gamma]}} x'$ then $f(x) \mathrel{E^{[\Gamma \times \Delta]}} f(x')$. Conversely, suppose $f(x) \mathrel{E^{[\Gamma \times \Delta]}} f(x') $ and let $(\gamma,\delta)$ be such that $(\gamma, \delta) \cdot f(x) \mathrel{E^{\Gamma \times \Delta}} f(x')$. Then we must have $\delta=1_{\Delta}$, so that $\gamma \cdot x \mathrel{\left(E^{\omega}\right)^{\Gamma}} x'$ and hence $x \mathrel{\left(E^{\omega}\right)^{[\Gamma]}} x'$.
\end{proof}

In particular, if $\Gamma$ has a subgroup isomorphic to $\Gamma \times \Delta$ for some infinite group $\Delta$, then $\left(E^{\omega}\right)^{[\Gamma]} \leq_B E^{[\Gamma]}$.

Another property of equivalence relations is that of being \emph{pinned}, which is defined using forcing. We briefly recall the definition; we refer the reader to \cite{kanovei} or \cite{zap-pinned} for properties of pinned equivalence relations.

\begin{defn}
Let $E$ be an analytic equivalence relation on $X$. A \emph{virtual $E$-class} is a pair $\langle \PP, \tau \rangle$ where $\PP$ is a poset and $\tau$ is a $\PP$-name so that $\Vdash_{\PP \times \PP} \tau_{\ell} \mathrel{E} \tau_r$, where $\tau_{\ell}$ and $\tau_r$ are the interpretations of $\tau$ using the left and right generics, respectively. A virtual $E$-class $\langle \PP, \tau \rangle$ is \emph{pinned} if there is some $x \in X$ from the ground model so that $\Vdash \tau \mathrel{E} \check{x}$. We say that $E$ is \emph{pinned} if every virtual $E$-class is pinned.
\end{defn}

\begin{thm}
  \label{thm:pinned}
  If $E$ is pinned then $E^{[\Gamma]}$ is pinned.
\end{thm}

\begin{proof}
  Suppose $E$ is pinned and let $\langle \PP, \tau \rangle$ be a virtual $E^{[\Gamma]}$-class. Then $\Vdash_{\PP\times\PP}\tau_\ell\mathrel{E}^{[\Gamma]}\tau_r$. Hence $\Vdash(\exists\gamma)(\forall \lambda)\tau_\ell(\lambda)\mathrel{E}\tau_r(\gamma\lambda)$. We can therefore find a condition $(p,q)$ and a $\gamma\in\Gamma$ such that
  \[(p,q)\Vdash(\forall\lambda)\,\tau_\ell(\lambda)
  \mathrel{E}\tau_r(\check\gamma\lambda)\text{.}
  \]
  Then temporarily considering the forcing $\PP^3$ and the three factor terms $\tau_0,\tau_1,\tau_2$, we have:
  \[(p,q,p)\Vdash(\forall\lambda)\tau_0(\lambda)
  \mathrel{E}\tau_1(\check\gamma\lambda)
  \mathrel{E}\tau_2(\lambda)\text{.}
  \]
  Since $E$ is transitive this implies that
  \[(p,p)\Vdash(\forall\lambda)\tau_\ell(\lambda)
  \mathrel{E}\tau_r(\lambda)\text{.}
  \]
  
  Now since $E$ is pinned we can find $x\in X^\Gamma$ such that for all $\lambda$ we have $p\Vdash\tau(\lambda)\mathrel{E}\check x(\lambda)$. It follows that $p\Vdash\tau\mathrel{E}^{[\Gamma]}\check x$. Finally we claim this is in fact unconditionally forced. Indeed if $q$ is any other condition then $(p,q)\Vdash\tau_\ell\mathrel{E}^{[\Gamma]}\check x\wedge\tau_\ell\mathrel{E}^{[\Gamma]}\tau_r$. Again using the transitivity of $E$, we conclude that $q\Vdash\tau\mathrel{E}^{[\Gamma]}\check x$ too.
\end{proof}

By contrast, the Friedman--Stanley jump does not preserve pinned-ness, as the relation $\Delta(\R)^{+} \sim_B F_2$ is not pinned (see 17.1.3 of \cite{kanovei}). Kanovei also shows in \cite{kanovei} that pinned-ness is preserved under Fubini products, so that the Louveau jump with the respect to the Fr\'{e}chet filter does preserve pinned-ness.

\begin{remark}
We feel that Corollary~\ref{cor:preservation} and Theorem~\ref{thm:pinned} justify the following Proclamation:
The $\Gamma$-jump is a kinder, gentler jump operator than the Louveau jump or the Friedman--Stanley jump.
\end{remark}

We may view the Friedman--Stanley jump and the $\Gamma$-jumps as special cases of a more general construction. Let $E$ be an equivalence relation on $X$, and let $(G,A)$ be a permutation group of a countable set $A$. We define the jump $E^{[G,A]}$ on $X^A$ by
\[x\mathrel{E}^{[G,A]} x'\iff (\exists g\in G)(\forall a\in A)\;x(g(a))\mathrel{E}x'(a)\text{.}
\]
The permutation group $S_\infty=(\Aut(\N),\N)$ corresponds to the Friedman--Stanley jump. In general the $(G,A)$-jump of a Borel equivalence relation need not be Borel; for example if $E$ has at least two equivalence classes, then the $(\Aut(\Q),\Q)$-jump of $E$ is Borel complete. To see this, observe that we can reduce isomorphism of countable linear orders by embedding a countable linear order as a suborder of $\Q$ in $X^{\Q}$ using two distinct $E$-classes, and isomorphism of countable linear orders is Borel-complete by Theorem 3 of \cite{friedman-stanley}. One may ask which (uncountable) group actions on $\N$ give rise to various types of equivalence relations, e.g., Borel complete relations, Borel equivalence relations, or the Friedman--Stanley jump. In \cite{Allison-Aristotelis}, the authors investigate this generalization, studying a $P$-jump operator for a Polish group $P$ of permutations of $\N$.

We may also ask how $\Gamma$-jumps for different groups $\Gamma$ compare to one another.

\begin{question}
  Given a fixed $E$, how many distinct complexities can arise as $E^{[\Gamma]}$?
For countable groups $\Gamma_1$ and $\Gamma_2$, when is $E^{[\Gamma_1]}\leq_B E^{[\Gamma_2]}$? 
\end{question}

Here, Shani \cite{shani} has characterized strong ergodicity between $\Gamma$-jumps of countable Borel equivalence relations in terms of group-theoretic properties, showing in particular:

\begin{thm}[Corollary 1.4 of \cite{shani}] 
  Let $E$ be a generically ergodic countable Borel equivalence relation. Then:
  \begin{itemize}
    \item $E^{[\Z]} <_B E^{[\Z^2]} <_B E^{[\Z^3]} <_B \cdots <_B E^{[\Z^{< \omega}]} <_B E^{[\F_2]}$;
    \item $E^{[\Z]}$ and $E^{[\Z_2^{< \omega}]}$ are $\leq_B$-incomparable.
  \end{itemize}
\end{thm} 

This contrasts sharply with the case of group actions, where actions of any two infinite countable abelian groups produce hyperfinite equivalence relations. It also shows that there is no ``least'' $\Gamma$-jump, although the $\F_2$-jump is the most complicated. We do not know whether incomparability can be extended to iterated jumps, such as through some notion of rigidity. Note, for instance, that by Lemma~\ref{lem:Zk} we have $E^{[\Z^2]} \leq_B \J[E]{\Z}{3}$.

\begin{question}
Are there countable groups $\Gamma_1$ and $\Gamma_2$ so that $\J[E]{\Gamma_1}{\alpha}$ and $\J[E]{\Gamma_2}{\beta}$ are $\leq_B$-incomparable for all $\alpha, \beta > 0$?
\end{question}

We also introduce two restrictions of Bernoulli jumps.

\begin{defn}
Let $E$ be an equivalence relation on $X$ and $\Gamma$ a countable group. The \emph{free part} of $X^{\Gamma}$ and the \emph{pairwise-inequivalent part} of $X^{\Gamma}$ are given by:
\begin{align*}
X^{\Gamma}_{\text{free}} &= \{ \bar{x} \in X^{\Gamma} : \forall \gamma (\gamma \neq 1_{\Gamma} \rightarrow \gamma \cdot \bar{x} \not\mathrel{E^{\Gamma}} \bar{x}) \} \\
X^{\Gamma}_{\text{p.i.}} &= \{ \bar{x} \in X^{\Gamma} : \forall \gamma, \delta ( \gamma \neq \delta \rightarrow \bar{x}(\gamma) \not\mathrel{E} \bar{x}(\delta)) \} .
\end{align*}
We let $E^{[\Gamma]}_{\text{free}} = E^{\Gamma} \upharpoonright X^{\Gamma}_{\text{free}}$ and $E^{[\Gamma]}_{\text{p.i.}} = E^{\Gamma} \upharpoonright X^{\Gamma}_{\text{p.i.}}$.
\end{defn}

We immediately have $E^{[\Gamma]}_{\text{p.i.}} \leq_B E^{[\Gamma]}_{\text{free}} \leq_B E^{[\Gamma]}$. We will see below that in certain cases all three are bireducible, but this is not true in general. Indeed, this fails already for $E=\Delta(2)$, as $\Delta(2)^{[\F_2]}$ is the shift equivalence relation $E(\F_2,2)$ which is bireducible with the universal countable Borel equivalence relation $E_{\infty}$ by Proposition~1.8 of \cite{dougherty-jackson-kechris}, whereas $\Delta(2)^{[\F_2]}_{\text{free}}$ is the free part of the shift, which is bireducible with the universal treeable countable Borel equivalence relation $E_{\infty T}$ and hence strictly below (see, e.g., Theorem~3.17 and Corollary~3.28 of \cite{jackson-kechris-louveau}).

\begin{question}
For which $E$ and $\Gamma$ do we have $E^{[\Gamma]} \leq_B E^{[\Gamma]}_{\text{free}}$ or $E^{[\Gamma]}_{\text{free}} \leq_B E^{[\Gamma]}_{\text{p.i.}}$?
\end{question}

\section{Comparing $\Gamma$-jumps to Friedman--Stanley jumps}
\label{sec:comparing}

We begin by recalling that the Friedman--Stanley tower $F_{\alpha}$ for $\alpha< \omega_1$ is defined analogously to the iterated $\Gamma$-jump.

\begin{defn}
The equivalence relation $F_{\alpha}$ for $\alpha< \omega_1$ is defined recursively by
\begin{align*}
F_0 &= \Delta(\omega), \\
F_{\alpha+1} &= F_{\alpha}^{+}, \\
F_{\lambda} &= \left(\bigoplus_{\alpha<\lambda}F_{\alpha}\right)^{+} \text{for $\lambda$ a limit} .
\end{align*}
Equivalently, $F_{\alpha}$ is the isomorphism relation on countable well-founded trees on $\omega$ of height at most $1+\alpha$.
\end{defn}

We compare the $\Z$-jump to the Friedman--Stanley jump.

\begin{defn}
  We say that $E$ is \emph{weakly absorbing} if $E$ has perfectly many classes and $\left( E^{< \omega}\right)^{+} \leq_B E^{+}$.
\end{defn}

Note that if $E$ has perfectly many classes and $E \times E \leq_B E$, then $E$ is weakly absorbing, so in particular this holds for $E_0$, $E_{\infty}$, and $F_{\alpha}$ for all $\alpha \geq 1$.

\begin{lem}
For any  $E$ with at least two classes, $E^{[\Z]}$ is  weakly absorbing.
\end{lem}

\begin{proof}
When $E$ has at least two classes,  $E^{[\Z]}$ is above $E_0$ and hence has perfectly many classes. We define a reduction from $\left( \left(E^{[\Z]}\right)^{< \omega}\right)^{+}$ to $\left(E^{[\Z]}\right)^{+}$ as follows. Fix $E$-inequivalent points $z_0$ and $z_1$. Given $\{ ( x^{m,0},\ldots,x^{m,k_m}) : m \in \omega \}$, let it map to the set $\{y^m_{i_1,\ldots,i_{k_m}} :m \in \omega, i_1,\ldots, i_{k_m} \in \Z \}$, where 
\begin{align*} 
y^m_{i_1,\ldots,i_{k_m}} = \ldots, & z_0, z_1, x^{m,0}_0, x^{m,0}_0, x^{m,1}_{i_1}, x^{m,1}_{i_1},\ldots, x^{m,k_m}_{i_{k_m}}, x^{m,k_m}_{i_{k_m}}, z_0, z_1,   \\
& x^{m,0}_1, x^{m,0}_1, x^{m,1}_{1+i_1}, x^{m,1}_{1+i_1},\ldots, x^{m,k_m}_{1+i_{k_m}}, x^{m,k_m}_{1+i_{k_m}}, z_0, z_1, \ldots ,
\end{align*}
i.e., $y^m_{i_1,\ldots,i_{k_m}}$ consists of a $\Z$-sequence of blocks, indexed by $k \in \Z$, of the form $x^{m,0}_k, x^{m,0}_k, x^{m,1}_{k+i_1}, x^{m,1}_{k+i_1},\ldots, x^{m,k_m}_{k+i_{k_m}}, x^{m,k_m}_{k+i_{k_m}}$, separated by the pair $z_0$, $z_1$. If 
\[( x^{m,0},\ldots,x^{m,k_m}) \mathrel{\left(E^{[\Z]}\right)^{< \omega}} ( \tilde{x}^{\tilde{m},0},\ldots,\tilde{x}^{\tilde{m},\tilde{k}_{\tilde{m}}})\]
 then $k_m=\tilde{k}_{\tilde{m}}$ and there are $j_0,\ldots,j_{k_m}$ so that 
 \[ x^{m,0}_j \mathrel{E} \tilde{x}^{\tilde{m},0}_{j+j_0} \ \wedge\  x^{m,1}_j \mathrel{E} \tilde{x}^{\tilde{m},1}_{j+j_1}\ \wedge\ \ldots\ \wedge\ x^{m,k_m}_j \mathrel{E} \tilde{x}^{\tilde{m},k_m}_{j+j_{k_m}} \]
 for all $j$, and hence $y^m_{i_1,\ldots,i_{k_m}} \mathrel{E^{[\Z]}} \tilde{y}^{\tilde{m}}_{i_1+j_1-j_0,\ldots, i_{k_m}+j_{k_m}-j_0}$. Thus the function defined is a homomorphism. Conversely, because of the repetition in the blocks, we can recover the $\left(E^{[\Z]}\right)^{<\omega}$-class of $( x^{m,0},\ldots,x^{m,k_m})$ from any $y^m_{i_1,\ldots,i_{k_m}}$, so that it is also a cohomomorphism and hence a reduction.
\end{proof}

\begin{lem}
\label{lem:freepiz}
If $E$ has perfectly many classes then $E^{[\Z]}_{\mathrm{free}} \leq_B \left( E^{< \omega} \right)^{[\Z]}_{\mathrm{p.i.}}$.
\end{lem}

\begin{proof}
Given a non-periodic $\Z$-sequence $x$ of $E$-representatives, we construct a new sequence $y$ as follows. For each $n\in\Z$ we first define a real $p_n$ by $p_n(k,l)=1$ iff $x_{n+k}\mathrel{E} x_{n+l}$. We then let $d_n$ be the least $d>0$  such that  $x_n\mathrel{E} x_{n+d}$ and $p_{n}=p_{n+d}$, if such exists, and $d_n=0$ otherwise. Finally we let $y_n = ( p_n, x_n,\ldots,x_{n+d_n} )$. Thus each coordinate $y_n$ has several pieces, and we say that two such coordinates $y_n$ and $y_{n'}$ are equivalent if they are equivalent with respect to $\Delta(\R)\times E^{<\omega}$.
  
  We claim that the entries $y_n$ are pairwise inequivalent. If this is not the case, we can find $n_0<n_1$ such that $y_{n_0}$ is equivalent to $y_{n_1}$. Let $d=d_{n_0}=d_{n_1}$. Note that $d>0$. Then the blocks $x_{n_0},\ldots,x_{n_0+d}$ and $x_{n_1},\ldots,x_{n_1+d}$ are pointwise $E$-equivalent. 
  Next, since we have $x_{n_0+i} \mathrel{E} x_{n_1+i}$ for $i\leq d$, and since $p_{n_0+d}=p_{n_0}$, we can conclude $x_{n_0+d+i}\mathrel{E} x_{n_1+d+i}$. Using this reasoning inductively, we can conclude that $x_{n_0+kd+i} \mathrel{E} x_{n_1+kd+i}$ for all $k>0$. The fact that $p_{n_0+d}=p_{n_0}$ can be used right-to-left to obtain the same for $k$ negative as well. In other words, $x$ is periodic with period $n_1-n_0$. This contradicts our assumption, and completes the claim.
  
  Thus the map $x \mapsto y$ is a reduction from the free part of $E^{[\Z]}$ to the pairwise inequivalent part of $(\Delta(\R)\times E^{<\omega})^{[\Z]}$. Since $E$ has perfectly many classes we have $\Delta(\R)\times E^{<\omega} \leq_B E^{< \omega}$, so $E^{[\Z]}_{\text{free}} \leq_B \left( E^{< \omega} \right)^{[\Z]}_{\text{p.i.}}$. 
\end{proof}

\begin{thm}
\label{thm:EZtoEplus}
 If $E$ is weakly absorbing then $E^{[\Z]} \leq_B E^{+}$.
\end{thm}

\begin{proof}
Since $E$ has perfectly many classes, we may fix $a,b \in X$ with $a \not\mathrel{E} b$ so that $E \leq_B E \upharpoonright X \setminus [\{a,b\}]$. From the previous lemma we then have that there is a reduction $g$ from $E^{[\Z]}_{\text{free}}$ to $\left( (E \upharpoonright X \setminus [\{a,b\}]) ^{< \omega} \right)^{[\Z]}_{\text{p.i.}}$.
Let $P=\bigcup_{k \geq 1} P_k$ be the set of periodic elements, where $P_k$ consists of $x \in X^\Z$ such that for all $n$ we have $x_n \mathrel{E} x_{n+k}$. Let $k(x)$ be the least $k \geq 1$ so that $x \in P_k$ if $x \in P$, and $k(x)=0$ if $x \notin P$. We now define a reduction $f$ from $E^{[\Z]}$ to $\left(E^{< \omega}\right)^{+}$ as follows; since $E$ is weakly absorbing this will be sufficient.
If $k(x) >0$, we let $f$ map $x \in P_{k(x)}$ to $\{ (a, x_i,x_{i+1},\ldots,x_{i+k(x)-1}) : 0 \leq i < k(x) \}$. If $k(x)=0$, let $f$ map $x$ to  $\{  g(x)_n {}^\smallfrown b ^\smallfrown g(x)_{n+1} : n \in \Z \}$, noting that the concatenation map $(z,w) \mapsto z ^\smallfrown b ^\smallfrown w$ provides a reduction from $\left( (E \upharpoonright X \setminus [\{a,b\}]) ^{< \omega} \right)^2$ to $E^{< \omega}$. 

Suppose first that $x \mathrel{E}^{[\Z]} x'$. If  $x \in P$ then $x'\in P$ and $k(x)=k(x')$, so that 
\begin{multline*}
\{ [(a, x_i,x_{i+1},\ldots,x_{i+k(x)-1})]_{E^{<\omega}} : 0 \leq i < k(x) \} = \\
\{ [(a, x'_i,x'_{i+1},\ldots,x'_{i+k(x')-1})]_{E^{<\omega}} : 0 \leq i < k(x') \}.
\end{multline*}
If $x, x' \notin P$, then there is $m \in \Z$ with $g(x)_n \mathrel{E^{<\omega}} g(x')_{m+n}$ for all $n$. Then 
\begin{align*}
\{ [\langle g(x)_n,b,g(x)_{n+1} \rangle]_{E^{< \omega}} : n \in \Z \} &= \{ [\langle g(x')_{m+n},b,g(x')_{m+n+1} \rangle]_{E^{< \omega}} : n \in \Z \} \\
&= \{ [\langle g(x')_n,b,g(x')_{n+1} \rangle]_{E^{< \omega}} : n \in \Z \},
\end{align*}
i.e., $f(x) \mathrel{(E^{<\omega})^{+}} f(x')$. 

Suppose conversely that $f(x) \mathrel{(E^{<\omega})^{+}} f(x')$. We must have either both $x,x' \in P$ or both $x,x' \notin P$. If $x, x' \in P$ then we must have $k(x)=k(x')$, and there are $i$ and $i'$ so that $x_{i+j} \mathrel{E} x'_{i'+j}$ for $0 \leq j < k(x)$, and hence $x_{i+j} \mathrel{E} x'_{i'+j}$ for all $j$, so $x \mathrel{E^{[\Z]}} x'$. Suppose then $x,x' \notin P$. Then
$\{ [\langle g(x)_n,b,g(x)_{n+1} \rangle]_{E^{< \omega}} : n \in \Z \} = \{ [\langle g(x')_n,b,g(x')_{n+1} \rangle]_{E^{< \omega}} : n \in \Z \}$, so for each $n$ there is an $m(n)$ with $g(x)_n \mathrel{E^{<\omega}} g(x')_{m(n)}$ and  $g(x)_{n+1} \mathrel{E^{<\omega}} g(x')_{m(n)+1}$. Because the $g(x')_m$'s are pairwise inequivalent, we conclude that $m(n+1)=m(n)+1$ for all $n$, so there is $m$ with $m(n)=n+m$ for all $n$. Thus  $g(x) \mathrel{(E^{<\omega})^{[\Z]}} g(x')$ so $x \mathrel{E}^{[\Z]} x'$.
\end{proof}

It turns out that $E^{[\Gamma]}$ is not reducible to $E^{+}$ in general. In particular the following is shown in \cite{Allison-Shani}:

\begin{thm}[Shani]
$\left(E_0^{\omega}\right)^{[\Z]}$ is not potentially $\bPi^0_3$. Hence, neither $\left(E_0^{\omega}\right)^{[\Z]}$ nor $\left(E_0\right)^{[\Z^2]}$ is reducible to $F_2$.
\end{thm}

Noting that $(E_0)^{+} \sim_B F_2$, we thus have that $\left(E_0\right)^{[\Z^2]} \not\leq_B \left(E_0\right)^{+}$.

\begin{question}
For which $E$ and $\Gamma$ is $E^{[\Gamma]} \leq_B E^{+}$?
\end{question}

A slight modification of the argument in \ref{thm:EZtoEplus} can be used to show:

\begin{lem}
  \label{lem:z2-free-part}
$\left(E_0\right)^{[\Z]} \leq_B \left(E_0\right)^{[\Z]} _{\text{p.i.}}$, i.e, $Z_2$ is Borel reducible to the pairwise inequivalent part of $Z_2$.
\end{lem}

\begin{proof}
Note that $E_0^{< \omega}$ is hyperfinite, as is any equivalence relation of finite index over $E_0^{< \omega}$, and hence reducible to $E_0$. Thus using Lemma~\ref{lem:freepiz}
we may fix pairwise $E_0$-inequivalent $\{a_n : n \in \Z\}$ and find a reduction $f$ from $\left(E_0\right)^{[\Z]}_{\text{free}}$ to $\left( E_0 \upharpoonright X \setminus A \right)^{[\Z]}_{\text{p.i.}}$, where $A=\bigcup_{n \in \Z} [a_n]$. We extend $f$ to the periodic part of $\left(E_0\right)^{[\Z]}$ as follows. With the notation from above, for $x \in P$ let 
\[ g(x)= \{ (x_i,x_{i+1},\ldots,x_{i+k(x)-1}) : 0 \leq i < k(x) \}.\]
 The relation that such finite sets consist of the same $E_0^{< \omega}$-classes is then hyperfinite, being of finite index over $E_0^{< \omega}$, so we may fix a reduction $h$ of $\left(E_0\right)^{[\Z]} \upharpoonright P$ to $E_0 \upharpoonright X \setminus A$. Then for $x \in P$ we let $f(x)(0)=h(x)$ and $f(x)(n)=a_n$ for $n \neq 0$ to complete the definition of $f$.
\end{proof}
The same technique can be applied to the $\Z$-jump of some other equivalence relations, such as $E_{\infty}$ and $F_2$, but we do not know if this is true for other $E$.

We can now compare the hierarchy $Z_\alpha$ with the hierarchy $F_\alpha$.

\begin{thm}
  \label{thm:ZalphaFalpha}
  For all $\alpha\geq2$, $Z_\alpha$ is Borel reducible to $F_\alpha$.
\end{thm}

\begin{proof}
  By induction on $\alpha$. The case of $\alpha=2$ follows from Theorem~\ref{thm:EZtoEplus} since $Z_2 \sim_B E_0^{[\Z]}$ and $F_2 \sim_B E_0^{+}$, and each iteration of the $\Z$-jump preserves the property of being weakly absorbing so we may again apply Theorem~\ref{thm:EZtoEplus}.
\end{proof}

As noted earlier, $\Gamma$-jumps are gentler than the Friedman--Stanley jump so the reverse of Theorem~\ref{thm:ZalphaFalpha} is false.

\begin{prop}
  $F_2$ is not Borel reducible to $\J{\Gamma}{\alpha}$ for any $\alpha < \omega_1$.
\end{prop}

\begin{proof}
  $\J{\Gamma}{\alpha}$ is pinned because the identity relation is pinned, and $\Gamma$-jumps and countable products of pinned relations are pinned. On the other hand, $F_2$ is not pinned and being pinned is preserved downward under $\leq_B$.
\end{proof}

As $E_0^{+} \sim_B E_{\infty}^{+} \sim_B F_2$, we get:

\begin{cor}
  $E_0^{[\Z]} <_B F_2$ and $E_{\infty}^{[\Z]} <_B F_2$.
\end{cor}

Although none of the $Z_{\alpha}$'s are above $F_2$, we will see below that they are unbounded in Borel complexity.

The result of Shani noted above shows that Theorem~\ref{thm:ZalphaFalpha} does not hold for all countable groups $\Gamma$, as we do not always have $\J{\Gamma}{2} \leq_B F_2$, e.g. for $\Gamma=\Z^2$ (since $\J{\Z^2}{1} \sim_B E_0$). Proposition~\ref{prop:sharp-bounds} below will show that this also does not hold for $\Gamma=\Z_2^{< \omega}$. Potential complexity bounds do allow us to give a weaker comparison, which Proposition~\ref{prop:sharp-bounds} will show to be the best possible in general.

We refer below to the equivalence relations $\cong^{\ast}_{\alpha}$, which are defined in Section~6 of \cite{hjorth-kechris-louveau}. These may be viewed as restricted forms of the relations $F_{\alpha}$; we do not need the specifics of their definitions, but note the key point that $\cong^{\ast}_{\alpha} <_B F_{\alpha}$ for all $\alpha$ for which they are defined. We also recall that for a pointclass $\Gamma$, the pointclass $D(\Gamma)$ consists of all sets which are the difference of two sets in $\Gamma$.

\begin{lem}
  $F_{\alpha}^{[\Gamma]} \leq_B \mathord{\cong}^{\ast}_{\alpha+1}$ for $\alpha \geq 2$, and hence $F_{\alpha}^{[\Gamma]} <_B F_{\alpha+1}$ for $\alpha > 0$. 
\end{lem}

\begin{proof}
  Theorem~2 of \cite{hjorth-kechris-louveau} shows that for a Borel equivalence relation $E$ induced by an action of a closed subgroup of $S_{\infty}$ we have $E \leq F_{\alpha}$ iff $E \in  \pot(\bPi^0_{\alpha+1})$ for $\alpha>0$ not a limit, and $E \leq F_{\alpha}$ iff $E \in  \pot(\bPi^0_{\alpha})$ for $\alpha$ a limit. Then $F_{\alpha}^{[\Gamma]}$ is still induced by an action of a closed subgroup of $S_{\infty}$ and is in $\pot(\bSigma^0_{\alpha+2}) $ for $\alpha$ not a limit or in $\pot(\bSigma^0_{\alpha+1})$ for $\alpha$ a limit. From Theorem~4.1 of \cite{hjorth-kechris-louveau} we then have that $F_{\alpha}^{[\Gamma]}  \in \pot(D(\bPi^0_{\alpha+1}))$ for $\alpha$ not a limit, and hence by Corollary~6.4 of \cite{hjorth-kechris-louveau} we have $F_{\alpha}^{[\Gamma]} \leq_B \mathord{\cong}^{\ast}_{\alpha+1} <_B F_{\alpha+1}$ in both cases for $\alpha \geq 2$. The case of $\alpha=1$ is immediate since $F_1^{[\Gamma]}$ is countable and hence strictly below $F_2$.
\end{proof}

From this we conclude:

\begin{prop}
  For any countable group $\Gamma$ and $\alpha \geq 2$ we have $\J{\Gamma}{\alpha} \leq_B \cong^{\ast}_{\alpha+1}$, so for any $\alpha< \omega_1$ we have $\J{\Gamma}{\alpha} <_B F_{\alpha+1}$.
\end{prop}

\begin{proof}
Note that $\J{\Gamma}{0} = \Delta(2) <_B \Delta(\R) = F_1$ and $\J{\Gamma}{1} <_B F_2$ since it is countable. We proceed by induction on $\alpha$. Successor steps follow immediately from the previous lemma. For a limit ordinal $\lambda$, if $\J{\Gamma}{\alpha} <_B F_{\alpha+1}$ for all $\alpha < \lambda$, then 
\[ \J{\Gamma}{\lambda} = \left(\bigoplus_{\alpha<\lambda}\J{\Gamma}{\alpha}\right)^{[\Gamma]} \leq_B \left(\bigoplus_{\alpha<\lambda}F_{\alpha+1} \right)^{[\Gamma]} \leq_B F_{\lambda}^{[\Gamma]} \leq_B \mathord{\cong}^{\ast}_{\lambda+1} \]
by the previous lemma.
\end{proof}

From Corollary~6.4 of \cite{hjorth-kechris-louveau}  we then have:

\begin{cor}
  \label{cor:upperbounds}
  For any countable group $\Gamma$ we have $\J{\Gamma}{0} \in \bDelta^0_1$, $\J{\Gamma}{1} \in \bSigma^0_2$, and:
  \begin{itemize}
    \item $\J{\Gamma}{\alpha} \in \pot( D(\bPi^0_{\alpha+1}))$ for $\alpha \geq 2$ not a limit;
    \item $\J{\Gamma}{\lambda} \in \pot(\bSigma^0_{\lambda+1})$ for $\lambda$ a limit. \qed
  \end{itemize}
\end{cor}

\section{$\Gamma$-trees}

\label{sec:Gammatree}

The Friedman--Stanley jump naturally corresponds to the group $S_{\infty}$ and its iterates correspond to isomorphism of well-founded trees. Namely, the equivalence relation $F_{\alpha}$ is bireducible with the isomorphism relation on countable well-founded trees of rank at most $2+\alpha$. We will see that $\Gamma$-jumps naturally correspond to certain group actions, and the iterates $\J{\Gamma}{\alpha}$ correspond to isomorphism of well-founded $\Gamma$-trees, which are trees where the children of each node carry the structure of a subset of $\Gamma$. We make this precise as follows. 

\begin{defn}
Let $\Gamma$ be a countable group. The language $\mathcal{L}_{\Gamma}$ consists of the binary relation $\prec$ together with binary relations $\{R_{\gamma} : \gamma \in \Gamma\}$. A \emph{$\Gamma$-tree} is an $\mathcal{L}_{\Gamma}$-structure which is a rooted tree with child relation $\prec$ and satisfies the additional $(\mathcal{L}_{\Gamma})_{\omega_1 \omega}$-formulas:
\begin{itemize}
\item $R_{\gamma}(u,v) \rightarrow \exists t (u \prec t \wedge v \prec t)$, for each $\gamma \in \Gamma$ with $\gamma \neq 1_{\Gamma}$
\item $R_{1_{\Gamma}}(u,v) \leftrightarrow u=v$
\item $\exists t (u \prec t \wedge v \prec t) \rightarrow \bigvee\limits_{\gamma \in \Gamma} R_{\gamma}(u,v)$
\item $\neg (R_{\gamma}(u,v) \wedge R_{\delta}(u,v))$, for each $\gamma \neq \delta \in \Gamma$
\item $(R_{\gamma}(u,v) \wedge R_{\delta}(v,w)) \rightarrow R_{\gamma \delta}(u,w)$, for each $\gamma,\delta \in \Gamma$
\end{itemize}
\end{defn}

We say that a $\Gamma$-tree is well-founded if it is well-founded as a tree, and we define rank in the usual way. Note that our definition ensures all $\Gamma$-trees are countable. Next we introduce the full $\Gamma$-tree, $T_{\Gamma}$. The automorphism group of $T_\Gamma$ will figure crucially in the next section.

\begin{defn}
  \label{defn:gamma-tree}
  Let $\Gamma$ be a countable group. The \emph{full $\Gamma$-tree}, denoted $T_\Gamma$, is the non-empty $\Gamma$-tree which additionally satisfies $(\forall t)\,(\exists u)\,(u \prec t)$, and for each $\gamma \in \Gamma$:
  \[(\forall t)\,(\forall u\prec t)\,(\exists v\prec t)\, R_{\gamma}(u,v)\text{.}
  \]
\end{defn}

This produces a categorical theory and defines the full $\Gamma$-tree uniquely up to isomorphism. For one model of $T_{\Gamma}$, we take the universe to be $\Gamma^{<\omega}$ and interpret each $R_{\gamma}$ by $R_{\gamma}(s ^\smallfrown \alpha,s ^\smallfrown\beta) \iff \alpha \gamma = \beta$.

\begin{defn}
We let $\cong^{\Gamma}$ denote the isomorphism relation on $\Gamma$-trees. For $\alpha < \omega_1$ we let $\cong^{\Gamma}_{\alpha}$ denote the isomorphism relation on well-founded $\Gamma$-trees of rank at most $1+\alpha$.
\end{defn}

Since every $\Gamma$-tree is isomorphic to a substructure of the full $\Gamma$-tree $T_{\Gamma}$, we have that $\cong^{\Gamma}$ is induced by an action of its automorphism group, $\Aut(T_{\Gamma})$. This group is isomorphic to the infinite wreath power $\Gamma^{\wr\omega}$, a cli group, and hence $\cong^{\Gamma}$ and any other $\Aut(T_{\Gamma})$-action is pinned.

We now can relate iterated $\Gamma$-jumps to isomorphism of well-founded $\Gamma$-trees. Nodes of rank 1 in a $\Gamma$-tree carry more structure than in a regular tree, where there are only countably many isomorphism types; hence the indexing differs by 1 from the case of the $F_\alpha$ and tree isomorphism. Also note that $\J{\Gamma}{0}$ is $\Delta(2)$, whereas $\cong^{\Gamma}_{0}$ is $\Delta(1)$. 

\begin{prop}
\label{prop:tree-jump}
For each $0<\alpha < \omega_1$, $\J{\Gamma}{\alpha}$ is Borel bireducible with $\cong^{\Gamma}_{\alpha}$.
\end{prop}

\begin{proof}
 For $\alpha=1$, $\J{\Gamma}{1}$ is the shift of $\Gamma$ on $2^{\Gamma}$, $E(\Gamma, 2)$. Given an element $x \in 2^{\Gamma}$ we naturally associate a $\Gamma$-tree $T(x)$ of rank 2 consisting of a root $v$ with children $\{ u_{\alpha} : x(\alpha)=1\}$, where we interpret $R_{\gamma}(u_{\alpha}, u_{\beta}) \iff \alpha \gamma = \beta$. If $x \mathrel{E(\Gamma, 2)} x'$, with $\gamma_0 \cdot x = x'$, then the map $f$ given by$f(v)=v'$ and $f(u_{\alpha})=u'_{\gamma_0 \alpha}$ is an isomorphism from $T(x)$ to $T(x')$. Conversely, let $f$ be an isomorphism from $T(x)$ to $T(x')$, so $f(v)=v'$. Fix $u=u_{\alpha} \in T(x)$ and suppose $f(u)=u'_{\alpha'}= u'_{\gamma_0 \alpha} \in T(x')$ for some $\gamma_0$. Then for all $\beta$ with $x(\beta)=1$ we have $R_{\alpha^{-1} \beta}(u_{\alpha}, u_{\beta})$, so $R_{\alpha^{-1} \beta}(u'_{\gamma_0 \alpha}, f(u_{\beta}))$; hence $f(u_{\beta})= u'_{\gamma_0 \beta}$. Thus $x(\beta)=1 \iff x'(\gamma_0 \beta)=1$, so $\gamma_0 \cdot x = x'$. Hence $x \mapsto T(x)$ witnesses $\J{\Gamma}{1} \leq_B \cong^{\Gamma}_1$. 
For the reverse reduction, given a rank 2 $\Gamma$-tree $T$, choose any non-root node $v_0 \in T$ and define $x(T) \in 2^{\Gamma}$ by $x(T)(\gamma) = 1 \iff \exists v \in T \ R_{\gamma}(v_0,v)$. Then $T(x(T)) \cong T$, so the map $T \mapsto x(T)$ witnesses $\cong^{\Gamma}_1 \leq_B \J{\Gamma}{1}$.

Induction steps are similar. Given $\J{\Gamma}{\alpha} \mathrel{\sim_B} \cong^{\Gamma}_{\alpha}$, with a reduction $f : \J{\Gamma}{\alpha} \leq_B \cong^{\Gamma}_{\alpha}$, we can send $x \in X^{\Gamma}$ to the tree $T(x)$ with children $\{u_{\gamma} : \gamma \in \Gamma\}$ of the root so that the subtree $T(x)_{u_{\gamma}} \cong f( x(\gamma))$ for each $\gamma$ to show $\J{\Gamma}{\alpha+1} \leq_B \cong^{\Gamma}_{\alpha+1}$. The reverse reduction is handled in an analogous manner, as are limit stages.
\end{proof}

Although we will see that in many instances the $\Gamma$-jump is proper on the Borel equivalence relations, we will always have that $\cong^{\Gamma}$ is a fixed point of the $\Gamma$-jump. This is analogous to the case of tree isomorphism, which is a fixed point of the Friedman--Stanley jump.

\begin{prop}
  \label{prop:iso-fixed}
  For a countable group $\Gamma$, $\cong^{\Gamma}$ is a fixed point of the $\Gamma$-jump, i.e., $(\cong^{\Gamma})^{[\Gamma]} \sim_B \mathord{\cong}^{\Gamma}$.
\end{prop}

\begin{proof}
Given $x \in X^{\Gamma}$ with each $x(\Gamma)$ coding a $\Gamma$-tree $T_{\gamma}$, we map $x$ to the $\Gamma$-tree $T(x)$ with children $\{u_{\gamma} : \gamma \in \Gamma\}$ of the root so that the subtree $T(x)_{u_{\gamma}} \cong T_{\gamma}$ to show $(\cong^{\Gamma})^{[\Gamma]} \leq_B \mathord{\cong}^{\Gamma}$.
\end{proof}

We also note that none of the $\cong^{\Gamma}$ are of maximal complexity among $S_{\infty}$-actions, since they are all pinned.

\begin{prop}
For every countable group $\Gamma$, $F_2 \not\leq_B \mathord{\cong}^{\Gamma}$, so $\cong^{\Gamma}$ is not Borel  complete.
\end{prop}

We will see in the next section that  $\cong^{\Gamma}$ is of maximal complexity among $\Aut(T_{\Gamma})$-actions.

\section{Reducing actions of $\Aut(T_{\Gamma})$ to iterated $\Gamma$-jumps}
\label{sec:reducing}

In this section we will establish that if $\Gamma$ is a countable group, then every Borel equivalence relation induced by a Borel action of a closed subgroup of the automorphism group of the full $\Gamma$-tree, $\Aut(T_\Gamma)$, is Borel reducible to some iterate $\J{\Gamma}{\alpha}$ of the $\Gamma$-jump.

In the next section we will use this, together with the fact that for certain groups $\Gamma$ the power $\Gamma^\omega$ has actions of cofinal essential complexity, to show that for such groups the $\Gamma$-jump is a proper jump operator. As a preview, note that by Theorem~2.3.5 of \cite{becker-kechris}, if $H$ is a closed subgroup of $\Aut(T_\Gamma)$ then every Polish $H$-space is reducible to a Polish $\Aut(T_\Gamma)$-space; similarly, we may reduce a Borel $H$-space to a Polish $H$-space. Since $\Gamma^{\omega}$ is isomorphic to a closed subgroup of $\Aut(T_\Gamma)$, if we know $\Gamma^{\omega}$ has actions of cofinal essential complexity, we will be able to conclude the iterates $\J{\Gamma}{\alpha}$ are properly increasing in complexity.

Recall from the previous section the full $\Gamma$-tree $T_\Gamma$ may be identified with $\Gamma^{< \omega}$. We let $T_\Gamma\restriction k$ be the restriction of $T_{\Gamma}$ to branches of length $k$. Then $\Aut(T_\Gamma\restriction k)$ is the wreath product of $k$-many copies of $\Gamma$, and $\Aut(T_\Gamma)$ is isomorphic the direct limit of the finite wreath powers of $\Gamma$. 

We are now ready to state the main theorem of this section.

\begin{thm}
\label{thm:Gactions}
Let $\Gamma$ be a countable group and let $E$ be a Borel equivalence relation induced by a continuous action of a closed subgroup of $\Aut(T_\Gamma)$. Let $\alpha<\omega_1$ be an ordinal such that $E$ is $\bPi^0_{\alpha}$.
\begin{enumerate}
\item If $\alpha = n < \omega$, $n \geq 3$, then $E \leq_B \J{\Gamma}{\omega\cdot(n-2)+1}$. 
\item If $\alpha = \lambda +1$, $\lambda$ a limit, then $E \leq_B \J{\Gamma}{\omega\cdot\lambda+1}$. 
\item If $\alpha = \lambda +n$, $\lambda$ a limit, $n \geq 2$, then $E \leq_B \J{\Gamma}{\omega\cdot(\lambda+n-2)+1}$. 
\end{enumerate}
\end{thm}

The proof of this theorem is based on some concepts and techniques from \cite{hjorth-kechris-louveau}, and we begin by recalling the relevant portions of this work, adapted slightly to our setting. Fix a countable group $\Gamma$ and a closed subgroup $H$ of $\Aut(T_\Gamma)$. Let $E=E_H^X$ be a $\bPi^0_{\alpha}$ orbit equivalence relation given by a continuous action of $H$ on a Polish space $X$.

The next definition concerns codes for $\bPi^0_{\alpha}$ subsets of $X \times X$. First, we fix for the remainder of this section an open basis $\{W_n\}$ for $X$.

\begin{defn}
A \emph{$\bPi^0_{\alpha}$-code} is a pair $(T,u)$ where $T$ is a well-founded tree of rank $\alpha$ so that if $t \in T$ is not terminal then $t ^\smallfrown n \in T$ for all $n$, and $u: \{t \in T: \text{$t$ is terminal} \} \rightarrow \omega \times \omega$. We write $u=(u_0,u_1)$. For $t \in T$, define the Borel set coded below $t$, $R_t$, by 
\[R_t = \begin{cases}
W_{u_0(t)} \times W_{u_1(t)}, & \text{if $t$ is terminal,} \\
\bigcap_n X \times X \setminus R_{t ^\smallfrown n}, & \text{otherwise.}
\end{cases} \]
The $\bPi^0_{\alpha}$ set coded by $(T,u)$ is $R_{\emptyset}$.
\end{defn}

Let $(T,u)$ be a $\bPi^0_{\alpha}$-code for $E$. For $t \in T$, let $|t|$ be the rank of $t$ in $T$. For  $t \in T_{\Gamma}$, let $\text{ht}(t)$ be its height in $T_{\Gamma}$, i.e., its length as a finite sequence.

We now define a convenient basis for $H$.

\begin{defn}
  Let $s$ be an enumeration of a finite partial function from $T_\Gamma$ to $T_\Gamma$ given by $s = \langle (d^s_i,e^s_i) : i < k \rangle$, where for each $i$ we have $\text{ht}(d^s_i)=\text{ht}(e^s_i)$ equal to the length of these sequences in $T_{\Gamma}$. We define $U_s = \{g \in H: \forall i < k \; g(d^s_i) = e^s_i\}$.
\end{defn}

The collection of all such $U_s$ is a countable basis for $H$. Note that $U_s$ may be empty, and the same partial function will have multiple enumerations. Let $B$ consist of all such $s$, and $B_k$ consist of just those $s\in B$ with maximum height of elements of the domain at most $k$.

\begin{defn}
  Given $s\in B$ we define the \emph{type} of $s$ to be the sequence $\langle \text{ht}(d^s_i) : i < k\rangle$ of heights of elements of the domain of $s$. Let $\mathcal{B}$ denote the set of types. For $C \in \mathcal{B}$, let $k_C$ be the size of the domain of all $s$ of type $C$, $d^C_i = \text{ht}(d^s_i)$ for $i<k_C$, and $j_C = \sum_{i<k_C} d^C_i$ the sum of heights of elements of the domain of $s$ (equivalently, heights of elements of the range). For types $C$ and $D$, we write $C \sqsubseteq D$ if the sequence of heights for $D$ extends that for $C$. Observe that if $s,t \in B$ and $U_t \subseteq U_s$, and $C$ is the type of $s$, then there is a type $D$ and $t'$ of type $D$ with $s \sqsubseteq t'$, $C \sqsubseteq D$, and $U_{t'}=U_t$.
\end{defn}

We next introduce an action of $\Aut(T_{\Gamma})$ on this basis.

\begin{defn}
  Let $a(g,s) = s^g$ be the action of $\Aut(T_\Gamma)$ on $B$ given by $s^g = \langle (g^{-1}(d^s_i),e^s_i) : i < k \rangle$, i.e., the corresponding partial function satisfies $s^g = s \circ g \upharpoonright g^{-1}[\dom(s)]$. Then $U_s \cdot g = U_{s^g}$. We use the same notation for the action of $\Aut(T_\Gamma\restriction k)$ on $B_k$. For $s,t \in B$ we write $s \sim t$ if $s$ and $t$ are in the same $\Aut(T_\Gamma)$-orbit (equivalently, the same $\Aut(T_\Gamma\restriction k)$-orbit for $s,t \in B_k$). If $s \sim t$ then $s$ and $t$ have the same type. Note, though, that $\Aut(T_\Gamma)$ acts on the enumerations, not the functions themselves, so two enumerations of the same partial function need not be $\sim$-equivalent.
\end{defn}

For $R \subseteq X \times X$ and $U, V$ open subsets of $H$, we define the double Vaught transform
\[ R^{\ast U, \ast V} = \{(x,y) : \forall^{\ast} g \in U \forall^{\ast} h \in V (g \cdot x, h \cdot y) \in R \}.\]
The following encodings are adapted from \cite{hjorth-kechris-louveau}, and are used to code the sections $R_t^{\ast U_s, \ast V}(x) = \{y : (x,y) \in R_t^{\ast U_s, \ast V}\}$.

\begin{defn}
  For $x \in X$, $t \in T$ with $|t| \geq 1$, and $s \in B$, define $N^s_t(x)$ by
  \[ N^s_t(x) =
    \begin{cases}
      \{ t ^\smallfrown n : x \in U_s^{-1} \cdot W_{u_0(t ^\smallfrown n)} \}, & \text{if $|t|=1$,} \\
      \{ N^r_{t ^\smallfrown n}(x): s \sqsubseteq r \in B \wedge n \in \omega \}, & \text{if $|t|>1$.}
    \end{cases}
  \]
\end{defn}

When $|t|=1$ this is essentially a real, and in general for $|t| = \alpha$ it is a hereditarily countable set of rank $\alpha$. Note that $t$ can be recovered from $N^s_t(x)$. As in \cite{hjorth-kechris-louveau} we have the following translation property:

\begin{lem}
  For $g \in G$, $N^s_t(g \cdot x) = N^{s^g}_t(x)$.
\end{lem}

\begin{defn}
Let $\tau_0$ be the topology of $X$. For any $x \in X$ and $\beta \leq |T|$ define the topology $\tau^x_{\beta}$ as the one generated by $\tau_0$ and the sets $R_t^{\ast U, \ast V}(x)$ for $1 \leq |t| \leq \beta$ and $U, V \in \{U_s : s \in B\}$.
\end{defn}
Then each $\tau^x_{\beta}$ is a Polish topology extending $\tau_0$. For a set $B \subseteq X$, let $B^{\Delta}=\{ x : \exists^{\ast} g (g \cdot x \in B)\}$. We summarize the key properties of these topologies from Sections 2 and 3 of \cite{hjorth-kechris-louveau}.

\begin{prop}[Hjorth--Kechris--Louveau]
Let $\mathcal{B}$ be an open basis for a topology on $X$. With the definitions above, we have:
\begin{itemize}
\item $\tau^x_{\beta}=\tau^{g \cdot x}_{\beta}$ for $g \in H$;
\item The action of $H$ on $X$ is continuous for $\tau^x_{\beta}$;
\item  If $x \mathrel{E} y$ then $\forall B \in \mathcal{B} (x \in B^{\Delta} \iff y \in B^{\Delta})$. The latter condition impies $\overline{[x]_E}=\overline{[y]_E}$;
\item If $[x]_E$ and $[y]_E$ are $G_{\delta}$ then $x \mathrel{E} y$ iff $\forall B \in \mathcal{B} (x \in B^{\Delta} \iff y \in B^{\Delta})$ iff $\overline{[x]_E}=\overline{[y]_E}$;
\item If $E$ is $\bPi^0_n$ for $3 \leq n < \omega$, then $[x]_E$ is $G_{\delta}$ in $\tau^x_{n-2}$, so $x \mathrel{E} y$ iff 
\[\tau^x_{n-2} = \tau^y_{n-2} \wedge \forall B \in \mathcal{B}(\tau^x_{n-2}) (x \in B^{\Delta} \iff y \in B^{\Delta})\text{;}
\]
\item The same holds with $\tau^x_{< \lambda}$ in place of $\tau^x_{n-2}$ when $E$ is $\bPi^0_{\lambda+1}$ for $\lambda$ a limit, and with $\tau^x_{\lambda+n-2}$ in place of $\tau^x_{n-2}$ when $E$ is $\bPi^0_{\lambda+n}$ for $\lambda$ a limit and $n \geq 2$.
\end{itemize}
\end{prop}

In case $E$ is $\bPi^0_2$ we have that $E$ is reducible to $\Delta(\R)$, so we begin with the case where $E$ is $\bPi^0_n$ for $3 \leq n < \omega$. We show that every $\bPi^0_n$ Polish $H$-space is reducible to some iterate of the $\Gamma$-jump of the identity relation, $\J{\Gamma}{\alpha}$. We note the modifications for cases for other $\alpha$ later.

\begin{defn}
  Define the following hereditarily countable sets:
  \begin{align*}
    A(x) &= \{ N^s_t(x) : |t| \leq n-2 \wedge s \in B \} \\
    B(x) &= \left\{ \langle m, \langle r_0, N^{s_0}_{t_0}(x) \rangle, \ldots, \langle r_{k-1} , N^{s_{k-1}}_{t_{k-1}}(x) \rangle \rangle : |t_i| \leq n-2 \vphantom{\left[ W_m \cap \bigcap_{i<k} R_{t_i}^{\ast U_{s_i}, \ast U_{r_i}}(x) \right]^{\Delta}} \right. \\
    & \qquad\qquad \left.\wedge r_i,s_i \in B \wedge x \in \left[ W_m \cap \bigcap_{i<k} R_{t_i}^{\ast U_{s_i}, \ast U_{r_i}}(x) \right]^{\Delta} \right\}.
  \end{align*}
\end{defn}

The idea is that $A(x)$ codes the topology $\tau_{n-2}^x$, and $B(x)$ codes the orbit closure of $x$ in this topology. Hjorth--Kechris--Louveau establish the following in Lemma~3.2 of \cite{hjorth-kechris-louveau}:

\begin{thm}[Hjorth--Kechris--Louveau]
  $x \mathrel{E} y$ if and only if $A(x)=A(y) \wedge B(x)=B(y)$.
\end{thm}

We will show that the equality of the sets $A(x)$ and $B(x)$ can be encoded into iterates of the $\Gamma$-jump. We first inductively define functions $c^s_t(x)$  and equivalence relations $F^C_t$ reducible to iterates of the $\Gamma$-jump in order to encode $N_t^s$. The bounds on the number of iterates follow from properties of the $\Gamma$-jump established earlier; we give only brief comment on the calculations in proofs in this section, and will discuss these bounds further in Section~\ref{sec:bounds}. We assume $|t| \geq 1$ for $t \in T$ unless stated otherwise.

\begin{lem}
  For $t\in T$ with $|t| \geq 1$, and $s \in B_k$ of type $C$, there is a Borel function $c^s_t(x)$ and an equivalence relation $F^C_t$, reducible to $\J{\Gamma}{\omega \cdot(|t|-1)}$ for $|t|>1$, so that for $g \in \Aut(T_\Gamma)$ and $r$ and $s$ of type $C$ we have:
  \begin{enumerate}
    \item If $y=g \cdot x$ then $c_t^{s^g}(x) \mathrel{F_t^C} c_t^s(y)$;
    \item If $c_t^r (x) \mathrel{F_t^C} c_t^s(y)$ then $N_t^r (x) = N_t^s(y)$.
  \end{enumerate}
\end{lem}

\begin{proof}
  (a) We use induction on $|t|$. For $|t|=1$, let $c^s_t(x) = N^s_t(x) \in \R$ and $F^C_t = \Delta(\R)$. If $y= g \cdot x$ then $c_t^{s^g}(x) = N_t^{s^g}(x) = N_t^s(g \cdot x) = N_t^s(y) =  c_t^s(y)$, and if $c_t^r (x) = c_t^s(y)$ then $N_t^r (x) = N_t^s(y)$.

  For $|t|>1$, define $F^C_t$ by
  \[ F^C_t = \prod_{n \in \omega} \prod _{D \in \mathcal{B} : C \sqsubseteq D} \left(F^{D}_{t ^\smallfrown n}\right)^{\left[ \Gamma^{2(j_D-j_C)}\right]} .
  \] 
  Next given $C \sqsubseteq D$ and a sequence $v = \alpha_{k_C} {}^\smallfrown \beta_{k_C} {}^\smallfrown \cdots {}^\smallfrown \alpha_{k_D-1} {}^\smallfrown \beta_{k_D-1} \in \Gamma^{2(j_D-j_C)}$ with $\alpha_i,\beta_i \in \Gamma^{d^D_i}$ for $j_C \leq i < j_D$, we let
  \[c^s_t(x)(n)(D)(v) = c^{s^\smallfrown v}_{t ^\smallfrown n}(x),
  \]
  where $s ^\smallfrown v = s ^\smallfrown (\alpha_{k_C}, \beta_{k_C}) {}^\smallfrown \cdots {}^\smallfrown (\alpha_{k_D-1}, \beta_{k_D-1})$.

  If $y = g \cdot x$, fix $n \in \omega$ and $D \in \mathcal{B}$ with $C \sqsubseteq D$. Then for each $v \in \Gamma^{2(j_D-j_C)}$ we have 
  \begin{align*}
    c_t^s(y)(n)(D)(v) &=\quad c_{t ^\smallfrown n}^{s ^\smallfrown v}(g \cdot x) \\
    & \mathrel{F_{t ^\smallfrown n}^D} c_{t ^\smallfrown n}^{ (s ^\smallfrown v)^g}(x) \\ 
    &=\quad c_{t ^\smallfrown n}^{s^g {}^\smallfrown v'}(x) \\
    &=\quad c_t^{s^g}(x)(n)(D)(v'),
  \end{align*}
  where $v' = g^{-1}(\alpha_{k_C}) {}^\smallfrown  \beta_{k_C} {}^\smallfrown \cdots {}^\smallfrown g^{-1}(\alpha_{k_D-1}) {}^\smallfrown \beta_{k_D-1}$.
  Letting $\tilde{g} \in \Gamma^{2(j_D-j_C)}$ be given by $\tilde{g} = g \restriction d^D_{k_C} {}^\smallfrown \mathrm{id} \restriction d^D_{k_C} {}^\smallfrown \cdots {}^\smallfrown g \restriction d^D_{k_D-1} {}^\smallfrown \text{id} \restriction d^D_{k_D-1}$ (where each concatenated factor acts on the corresponding coordinates) we have that $v' = \tilde{g}^{-1} \cdot v$ for all $v\in\Gamma^{2(j_D-j_C)}$. Thus $\tilde{g}$ witnesses 
  \[ c_t^{s^g}(x)(n)(D) \mathrel{\left(F^{D}_{t ^\smallfrown n}\right)^{\left[ \Gamma^{2(j_D-j_C)}\right]}} c_t^s(y)(n)(D),
  \]
  and hence $c_t^{s^g}(x) \mathrel{F_t^C} c_t^s(y)$.

  (b) If $c_t^r (x) \mathrel{F_t^C} c_t^s(y)$ then for all $n \in \omega$ and $C \sqsubseteq D \in \mathcal{B}$ we have 
  \[ c_t^r(x)(n)(D) \mathrel{\left(F^{D}_{t ^\smallfrown n}\right)^{\left[ \Gamma^{2(j_D-j_C)}\right]}} c_t^s(y)(n)(D).
  \]
  Thus there is $h=h_{n,D} \in \Gamma^{2(j_D-j_C)}$ such that for all $v \in \Gamma^{2(j_D-j_C)}$ we have
  \[ c_t^r(x)(n)(D)(h^{-1}\cdot v) \mathrel{F_{t ^\smallfrown n}^D} c_t^s(y)(n)(D)(v)
  \]
  and hence
  \[ c_{t ^\smallfrown n}^{r ^\smallfrown h^{-1}\cdot v}(x) \mathrel{F_{t ^\smallfrown n}^D} c_{t ^\smallfrown n}^{s ^\smallfrown v}(y),
  \]
  so by inductive assumption we have $N_{t ^\smallfrown n}^{r ^\smallfrown h^{-1}\cdot v}(x) = N_{t ^\smallfrown n}^{s ^\smallfrown v}(y)$. Thus
  \[ \{ N_{t ^\smallfrown n}^{r ^\smallfrown  v}(x) : v \in \Gamma^{2(j_D-j_C)} \} = \{ N_{t ^\smallfrown n}^{s ^\smallfrown v}(y) : v \in \Gamma^{2(j_D-j_C)} \} ,
  \]
  so
  \begin{align*}
    N_t^r(x)& = \bigcup_n \bigcup_{D \in \mathcal{B} : C \sqsubseteq D} \{N_{t ^\smallfrown n}^{r ^\smallfrown v}(x) : v \in \Gamma^{2(j_D-j_C)}\}  \\
   & =\bigcup_n \bigcup_{D \in \mathcal{B} : C \sqsubseteq D}  \{ N_{t ^\smallfrown n}^{s ^\smallfrown v}(y) : v \in \Gamma^{2(j_D-j_C)} \} = N_t^s(y) 
   \end{align*}
  as required.
  
To see that $F^C_t$ is reducible to $\J{\Gamma}{\omega \cdot(|t|-1)}$, note that for $|t|=1$ we have $F^C_t=\Delta(\R)$, and iterated jumps starting with $\Delta(\R)$ are equivalent to those starting with $\Delta(2)$ for $\alpha>0$, as noted earlier, so we may safely consider $F^C_t$ to be $\J{\Gamma}{0}$ for $|t|=1$. Assume now that $|t|>1$ and the bound holds for all $s \in T$ with $|s|<|t|$, so $F^C_{t ^\smallfrown n} \leq_B \J{\Gamma}{\omega \cdot(|t|-2)}$. Then 
\[ F^C_t = \prod_{n \in \omega} \prod _{D \in \mathcal{B} : C \sqsubseteq D} \left(F^{D}_{t ^\smallfrown n}\right)^{\left[ \Gamma^{2(j_D-j_C)}\right]} \leq_B \left( \bigoplus_{k \in \omega} \left(\J{\Gamma}{\omega \cdot(|t|-2)}\right)^{[\Gamma^k]} \right)^{\omega} . \]
This is reducible to $\left( \bigoplus_{k \in \omega} \J{\Gamma}{\omega \cdot(|t|-2)+k+1} \right)^{\omega}$ by Lemma~\ref{lem:Zk}, which is reducible to $\J{\Gamma}{\omega \cdot(|t|-1)}$ by  Proposition~\ref{prop:Eomega}.
\end{proof}

\begin{defn}
 For $t \in T$, let $A_t(x) = \{ N^s_t(x) : s \in B \}$.
\end{defn}

\begin{lem}
\label{lem:lemmaforAt}
  For $t \in T$, there is a Borel function $f_t$ and an equivalence relation $F_t$, reducible to $\J{\Gamma}{\omega \cdot |t|}$, so that:
  \begin{enumerate}
    \item If $x \mathrel{E} y$ then $f_t(x)  \mathrel{F_t} f_t(y)$;
    \item If $f_t(x) \mathrel{F_t} f_t(y)$ then $A_t(x) = A_t(y)$.
  \end{enumerate}
\end{lem}

\begin{proof}
  (a) Let $F_t = \prod_{C \in \mathcal{B}} \left(F^C_t\right)^{\left[\Gamma^{2j_C}\right]}$ and $f_t(x)(C)(v) = c^{s(v)}_t(x)$, where $s(v) = \langle (\alpha_i,\beta_i) : i < k_C\rangle$ for $v = \alpha_0 {}^\smallfrown \beta_0 {}^\smallfrown \cdots {}^\smallfrown \alpha_{k_C-1} {}^\smallfrown \beta_{k_C-1} \in \Gamma^{2j_C}$. Suppose $x \mathrel{E} y$, so there is $g \in H$ with $y = g \cdot x$.  Fix $C \in \mathcal{B}$, so we have $c_t^{s^g}(x) \mathrel{F_t^C} c_t^s(y)$ for all $s \in C$ from the previous lemma. Thus for each $v \in  \Gamma^{2j_C}$ we have 
  \begin{align*}
    f_t(y)(C)(v) &=\quad c_t^{s(v)}(g \cdot x) \\
    &\mathrel{F_t^C} c_t^{ (s(v))^g}(x) \\ 
    &=\quad c_t^{s(v')}(x) \\
    &=\quad f_t(x)(C)(v'),
  \end{align*}
  where $v' = g^{-1}(\alpha_0) ^\smallfrown  \beta_0 {}^\smallfrown \cdots {}^\smallfrown g^{-1}(\alpha_{k_C-1}) {}^\smallfrown \beta_{k_C-1}$. Letting $\tilde{g} \in \Gamma^{2j_C}$ be given by $\tilde{g} = g \restriction d^C_0 {}^\smallfrown \text{id} \restriction d^C_0 {}^\smallfrown \cdots {}^\smallfrown g \restriction d^C_{k_C-1} {}^\smallfrown \mathrm{id} \restriction d^C_{k_C-1}$ (where each concatenated factor acts on the corresponding coordinates) we have that $v' = \tilde{g}^{-1} \cdot v$ for all $v\in\Gamma^{2j_C}$. Thus $\tilde{g}$ witnesses 
  \[ f_t(x)(C) \mathrel{\left(F^C_t\right)^{\left[\Gamma^{2j_C}\right]}} f_t(y)(C),
  \]
  and hence  $f_t(x) \mathrel{F_t} f_t(y)$.

  (b) Suppose $f_t(x) \mathrel{F_t} f_t(y)$, so for each $C \in \mathcal{B}$ we have $f_t(x)(C) \mathrel{\left(F^C_t\right)^{[\Gamma^{2j_C}]}} f_t(y)(C)$. 
Thus there is $h=h_C \in \Gamma^{2j_C}$ so that for all $v \in \Gamma^{2j_C}$ we have
  \[ f_t(x)(C)(h^{-1}\cdot v) \mathrel{F_t^C} f_t(y)(C)(v)
  \]
  and hence 
  \[ c_t^{s(h^{-1}\cdot v)}(x) \mathrel{F_t^C} c_t^{s(v)}(y),
  \]
  so we have $N_t^{s(h^{-1}\cdot v)}(x) = N_t^{s(v)}(y)$. Thus
  \[ \{ N_t^s(x) : s \in C \} = \{ N_t^s(y) : s \in C \} ,
  \]
  so
  \[ A_t(x) = \bigcup_{C \in \mathcal{B}} \{ N_t^s(x) : s \in C \} 
   = \bigcup_{C \in \mathcal{B}} \{ N_t^s(y) : s \in C \} = A_t(y)
  \]
  as required.
\end{proof}

\begin{lem}
  There is a Borel function $f_A$ and an equivalence relation $F_A$, reducible to $\J{\Gamma}{\omega \cdot (n-2)}$, so that:
  \begin{enumerate}
    \item If $x \mathrel{E} y$ then $f_A(x)  \mathrel{F_A} f_A(y)$;
    \item If $f_A(x) \mathrel{F_A} f_A(y)$ then $A(x) = A(y)$.
  \end{enumerate}
\end{lem}

\begin{proof}
  Since $t$ can be recovered from $N^s_t(x)$ we have $A(x)=A(y)$ if and only if $A_t(x) = A_t(y)$ for each $t$. Letting $F_A = \prod_{t: |t|\leq n-2} F_t$ and $f(x)(t)=f_t(x)$ suffices.
\end{proof}

\begin{defn}
  For $m,k \in \omega$, $\bar{r}=(r_0,\ldots,r_{k-1}) \in B^k$, and $\bar{t}=(t_0,\ldots,t_{k-1}) \in T^k$, let
  \begin{align*}
    B^{m,k}_{\bar{r},\bar{t}} (x) &= \left\{ \langle N^{s_0}_{t_0}(x) , \ldots,  N^{s_{k-1}}_{t_{k-1}}(x)  \rangle : s_i \in B \vphantom{\wedge\ x \in \left[ W_m \cap \bigcap_{i<k} R_{t_i}^{\ast U_{s_i}, \ast U_{r_i}}(x) \right]^{\Delta} }\right.\\
& \qquad\qquad \left.\wedge\ x \in \left[ W_m \cap \bigcap_{i<k} R_{t_i}^{\ast U_{s_i}, \ast U_{r_i}}(x) \right]^{\Delta} \right\}.
  \end{align*}
Noting that $R_t^{\ast U_s, \ast V}(g \cdot x) = R_t^{\ast U_{s^g}, \ast V}(x)$, we have that $B(x)=B(y)$ if and only if for all $m, k, \bar{r}$, and $\bar{t}$ with $|t_i| \leq n-2$ we have $B^{n,k}_{\bar{r},\bar{t}} (x) = B^{n,k}_{\bar{r},\bar{t}} (y)$.
\end{defn}

\begin{lem}
  There is a Borel function $f_B$ and an equivalence relation $F_B$, reducible to $\J{\Gamma}{\omega \cdot (n-2)}$, so that:
  \begin{enumerate}
    \item If $x \mathrel{E} y$ then $f_B(x)  \mathrel{F_B} f_B(y)$;
    \item If $f_B(x) \mathrel{F_B} f_B(y)$ then $B(x) = B(y)$.
  \end{enumerate}
\end{lem}

\begin{proof}
For $m$, $k$, $\bar{r} \in B^k$, and $\bar{t} \in T^k$ with $|t_i| \leq n-2$ let 
\[ F^{m,k}_{\bar{r},\bar{t}} = \prod_{\bar{C} \in \mathcal{B}^k} \left(\left( \prod_{i<k} F^{C_i}_{t_i} \right) \times \Delta(2)\right)^{\left[\Gamma^{2(j_{C_0}+\cdots+j_{C_{k-1}})}\right]}, \]
and for $\alpha_i \in \Gamma^{2j_{C_i}}$ define $f^{m,k}_{\bar{r},\bar{t}}(x)$ by 
$f^{m,k}_{\bar{r},\bar{t}}(x)(\bar{C})(\alpha_0 {}^\smallfrown \ldots {}^\smallfrown \alpha_{k-1}) =$
\[
  \begin{cases}
    \langle c^{s(\alpha_0)}_{t_0}(x) , \ldots,  c^{s(\alpha_{k-1})}_{t_{k-1}}(x) , 1 \rangle, & \text{if $x \in  \left[ W_m \cap \bigcap_{i<k} R_{t_i}^{\ast U_{s_i}, \ast U_{r_i}}(x) \right]^{\Delta}$}\text{,} \\ 
    \langle c^{s(\alpha_0)}_{t_0}(x) , \ldots,  c^{s(\alpha_{k-1})}_{t_{k-1}}(x) , 0 \rangle, & \text{if not,}
  \end{cases}
\]
where $s(\alpha_i)$ is as in the proof of Lemma~\ref{lem:lemmaforAt}. Similar to there, if $y = g \cdot x$ then for each $\bar{C}$ we will have $\tilde{g} \in \Gamma^{2(j_{C_0}+\cdots+j_{C_{k-1}})}$ with 
\begin{multline*}
  f^{m,k}_{\bar{r},\bar{t}}(x)(\bar{C})(\tilde{g}^{-1}(\alpha_0 {}^\smallfrown \ldots {}^\smallfrown \alpha_{k-1})) \mathrel{ \left(\left( \prod_{i<k} F^{C_i}_{t_i} \right) \times \Delta(2)\right)^{\Gamma^{2(j_{C_0}+\cdots+j_{C_{k-1}})}}} \\
  f^{m,k}_{\bar{r},\bar{t}}(y)(\bar{C})(\alpha_0 {}^\smallfrown \ldots {}^\smallfrown \alpha_{k-1}) .
\end{multline*}
Similarly, if $f^{m,k}_{\bar{r},\bar{t}}(x) \mathrel{F^{m,k}_{\bar{r},\bar{t}}} f^{m,k}_{\bar{r},\bar{t}}(y)$ then for each $\bar{C}$ there will be $g$ witnessing 
\begin{multline*}
  \left\{ \langle N^{s_0}_{t_0}(x) , \ldots,  N^{s_{k-1}}_{t_{k-1}}(x)  \rangle : s_i \in C_i
  \wedge\ x \in \left[ W_m \cap \bigcap_{i<k} R_{t_i}^{\ast U_{s_i}, \ast U_{r_i}}(x) \right]^{\Delta}  \right\}  \\
 = \left\{ \langle N^{s_0}_{t_0}(y) , \ldots,  N^{s_{k-1}}_{t_{k-1}}(y)  \rangle : s_i \in C_i
  \wedge\ y \in \left[ W_m \cap \bigcap_{i=1}^k R_{t_i}^{\ast U_{s_i}, \ast U_{r_i}}(y) \right]^{\Delta}  \right\},
\end{multline*}
so $B^{m,k}_{\bar{r},\bar{t}} (x) = B^{m,k}_{\bar{r},\bar{t}} (y)$.
Finally, define $F_B$ by
\[ F_B = \prod_m \prod_k \prod_{\bar{r} \in B^k} \prod_{\bar{t} \in T^k : |t_i| \leq n-2}  F^{m,k}_{\bar{r},\bar{t}}
\]
and $f_B$ by $f_B(x)(m,k,\bar{r},\bar{t}) = f^{m,k}_{\bar{r}.\bar{t}}(x)$.
\end{proof}

We are now ready to conclude the proof of Theorem \ref{thm:Gactions}.

\begin{proof}[Proof of Theorem \ref{thm:Gactions}]
Let $E$ be $\bPi^0_{\alpha}$.
  \begin{enumerate}
    \item For $\alpha=n \geq 3$, the function $f(x)= (f_A(x), f_B(x))$ is a reduction of $E$ to $F_A \times F_B$ as shown above, so $E \leq_B \J{\Gamma}{\omega\cdot(n-2)} \times \J{\Gamma}{\omega\cdot(n-2)} \leq_B \J{\Gamma}{\omega\cdot(n-2)+1}$.
    \item For $\alpha = \lambda+1$, $\lambda$ a limit, we repeat the preceding argument using the topology $\tau^x_{< \lambda}$ in place of $\tau^x_{n-2}$.
    \item For $\alpha = \lambda +n$, $\lambda$ a limit and $n \geq 2$, we use the topology $\tau^x_{\lambda+n-2}$.\qedhere
  \end{enumerate}
\end{proof}

\begin{cor}
  \label{cor:gamma-omega-actions}
  Let $\Gamma$ be a countable group and let $E$ be a Borel equivalence relation induced by a continuous action of $\Gamma^{\omega}$. Then $E \leq_B \J{\Gamma}{\alpha}$ for some $\alpha < \omega_1$.
\end{cor}

In particular, the above shows that a $\bPi^0_3$ Polish $\Aut(T_{\Gamma})$-space is reducible to $\J{\Gamma}{\omega}$. We can improve this slightly in the case of $\bSigma^0_2$.

\begin{prop}
  Let $\Gamma$ be a countable group and let $E$ be a Borel equivalence relation induced by a continuous action of a closed subgroup of $\Aut(T_\Gamma)$. If $E$ is $\bSigma^0_2$ then $E \leq_B \bigoplus_{n \in \omega} \J{\Gamma}{n}$.
\end{prop}

\begin{proof}
  Since $\Aut(T_{\Gamma})$ is a closed subgroup of $S_{\infty}$,  $E$ is essentially countable by Theorem~3.8 of \cite{hjorth-kechris-countable}, so let $f$ be a reduction of $E$ to $E_{\infty}$. For each $x$, there is some $z \in [f(x)]_{E_{\infty}}$ so that the set $\{g \in \Aut(T_{\Gamma}) : f(g \cdot x)=z\}$ is nonmeager, and hence comeager in some $U_s$. Define the relation $P(x,C)$ on $X \times \mathcal{B}$ by 
  \[ P(x,C) \iff \exists s \in C\, \exists z \in [f(x)]_{E_{\infty}}\, \forall^{\ast} g \in U_s \;\; f(g \cdot x)=z .
  \]
  Then $\forall x \exists C P(x,C)$ and $P$ is Borel and $E$-invariant (i.e., if $P(x,C)$ and $x \mathrel{E} x'$ then $P(x',C)$), so there is an $E$-invariant Borel function $\Psi: X \rightarrow \mathcal{B}$ with $P(x,\Psi(x))$ for all $x$ (take $\Psi(x)$ to be the least $C$ in some fixed enumeration of the countable set $\mathcal{B}$ which satisfies $P(x,C)$). Let $X_C = \Psi^{-1}(\{C\})$. On $X_C$ define the reduction $\varphi : E \upharpoonright X_C \leq_B \Delta(\R)^{\left[\Gamma^{2j_C}\right]}$ by
  \[ \varphi(x)(h) =
    \begin{cases}
      z,    & \text{if $\forall^{\ast} g \in U_{s(h)}\; f(g \cdot x ) = z$}\text{,}\\
      \ast, & \text{if no such $z$ exists,}
    \end{cases}
  \]
  where $s(h)$ is defined as in the proof of Lemma~\ref{lem:lemmaforAt} and $\ast$ is some new element $E_{\infty}$-inequivalent to all $z$.

  Suppose $x \mathrel{E} y$, so there is $g$ with $y = g \cdot x$. Then as in the proof of Lemma~\ref{lem:lemmaforAt} there is $\tilde{g} \in \Gamma^{ 2j_C}$ so that for all $h \in \Gamma^{2j_C}$ we have $s(\tilde{g}^{-1}\cdot h)=s(h)^{\tilde{g}^{-1}}$, so that $\tilde{g}$ witnesses $\varphi(x) \mathrel{\Delta(\R)^{\left[\Gamma^{2j_C}\right]}} \varphi(y)$. Conversely, if $\varphi(x) \mathrel{\Delta(\R)^{\left[\Gamma^{2j_C}\right]}} \varphi(y)$, then there are some $h$, $h'$, and $z$ with $\varphi(x)(h)=z \neq \ast$ and $\varphi(y)(h')=z$. Then $f(x) \mathrel{E_{\infty}} z \mathrel{E_{\infty}} f(y)$, so $x \mathrel{E} y$.
\end{proof}

\begin{cor}
\label{cor:countableZa}
  Let $E$ be a countable Borel equivalence relation. If $E \leq_B \J{\Gamma}{\alpha}$ for some $\alpha < \omega_1$ then $E \leq_B \bigoplus_{n \in \omega} \J{\Gamma}{n}$.
\end{cor}

\begin{proof}
  Let $\varphi$ be a reduction of $E$ to $\J{\Gamma}{\alpha}$. Then the $\J{\Gamma}{\alpha}$-saturation $B$ of the range of $\varphi$ is Borel by the Lusin--Novikov Theorem, and $E \sim_B \J{\Gamma}{\alpha} \upharpoonright B$. Thus $E \sim_B \J{\Gamma}{\alpha} \upharpoonright B$ is potentially $\bSigma^0_2$ by Theorem~3.8 of \cite{hjorth-kechris-countable}, so by Theorem~5.1.5 of \cite{becker-kechris} we can refine the topology of $B$ so that $\J{\Gamma}{\alpha} \upharpoonright B$ becomes a Polish $\Aut(T_{\Gamma})$-space which is $\bSigma^0_2$.  Thus $\J{\Gamma}{\alpha} \upharpoonright B$, and hence $E$, is reducible to $\bigoplus_{n \in \omega} \J{\Gamma}{n}$.
\end{proof}

Note that, e.g., when $\Gamma = \F_2$ we already have that $\J{\Gamma}{1} \sim_B E_{\infty}$, since this is the shift equivalence relation $E(\F_2,2)$ which is bireducible with $E_{\infty}$ by Proposition~1.8 of \cite{dougherty-jackson-kechris}. Thus the above bound is not always optimal, but we do not know if this is true for other $\Gamma$, such as $\Gamma=\Z$. The following seems quite optimistic:

\begin{question}
If $E$ is a countable Borel equivalence relation reducible to some $\J{\Gamma}{\alpha}$, is $E$ reducible to $\J{\Gamma}{2}$ or even to $\J{\Gamma}{1}$?
\end{question}

The same techniques will also show that isomorphism of $\Gamma$-trees is maximal among $\Aut(T_{\Gamma})$-actions. We introduce a useful augmentation of $\Gamma$-trees.

\begin{defn}
  A \emph{labelled $\Gamma$-tree} is the full $\Gamma$-tree $T_{\Gamma}$ with each node labelled with a real. We use $\cong^{\Gamma}_{\mathrm{lab}}$ to denote isomorphism of labelled $\Gamma$-trees.
\end{defn}

We may identify a labelled $\Gamma$-tree with a function from $T_{\Gamma}$ to $\R$. We can see that isomorphism of $\Gamma$-trees is bireducible with isomorphism of labelled $\Gamma$-trees.

\begin{lem}
  $ \mathord{\cong}^{\Gamma}_{\mathrm{lab}}  \mathrel{\sim_B} \mathord{\cong}^{\Gamma}$.
\end{lem}

\begin{proof}
  To see that isomorphism of $\Gamma$-trees is reducible to isomorphism of labelled $\Gamma$-trees, we can identify a $\Gamma$-tree $T$ with a subtree of $T_{\Gamma}$, and assign the labelling where a node $t \in T_{\Gamma}$ is labelled `1' if $t \in T$ and `0' if $t \notin T$.

  For the reverse reduction, let $g: \R \rightarrow \mathcal{P}(\Gamma \setminus \{1_{\Gamma}\})$ be an injection so that for $\alpha \neq \beta$ we have that $\{1_{\Gamma}\} \cup g(\alpha)$ and $\{1_{\Gamma}\} \cup g(\beta)$ are not isomorphic via a shift; such an injection exists since the shift action of $\Gamma$ on $2^{\Gamma}$ has perfectly many equivalence classes. Let $f\colon T_{\Gamma} \rightarrow\R$ be a labelled $\Gamma$-tree. For a node $t =(\alpha_0,\ldots, \alpha_{k-1}) \in T_{\Gamma}$, let $\tilde{t} = (\alpha_0, 1_{\Gamma}, \alpha_1, 1_{\Gamma}, \ldots, 1_{\Gamma}, \alpha_{k-1})$. Let $T(f)$ be the subtree of $T_{\Gamma}$ given by
  \[ T(f) = \{\tilde{t} : t \in T_{\Gamma} \} \cup \{\tilde{t} ^\smallfrown 1_{\Gamma} : t \in T_{\Gamma} \} \cup \{\tilde{t} ^\smallfrown \gamma : t \in T_{\Gamma} \wedge \gamma \in g(f(t)) \}.
  \]
  Then it is straightforward to check that the map $f \mapsto T(f)$ is a reduction from $\mathord{\cong}^{\Gamma}_{\text{lab}} $ to $\mathord{\cong}^{\Gamma}$.
\end{proof}

\begin{thm}
Let $G= \Aut(T_{\Gamma})$. For any Polish $G$-space $E_G^X$ we have $E_G^X \leq_B \mathord{\cong}^{\Gamma}$.
\end{thm}

\begin{proof}
It will suffice to show that $E_G^X \leq_B \mathord{\cong}^{\Gamma}_{\text{lab}}$. Fix an enumeration $\{ e_i : i \in \omega\}$ of $T_{\Gamma}$, and let $C_n$ be the type $\langle \text{ht}(e_0), \ldots, \text{ht}(e_{n-1}) \rangle$  for each $n$. 
For $t \in T_{\Gamma}$ with $\text{ht}(t)=j_{C_n}$ we let $s_t \in B$ be such that $t= d^{s_t}_0 {}^\smallfrown \cdots {}^\smallfrown d^{s_t}_{n-1}$ and $e^{s_t}_i=e_i$ for $i<n$. 
Given $x$, we let $T_x$ be the labelled $\Gamma$-tree defined as follows. For $t$ with $\text{ht}(t)=j_{C_n}$ for some $n$ we let $T_x(t) = \{\ell : x \in U_{s_t}^{-1} \cdot W_{\ell} \}$, and let $T_x(t) = \emptyset$ for other $t$. We claim that the map $x \mapsto T_x$ is a reduction from $E_G^X$ to $\cong^{\Gamma}_{\text{lab}}$.

Suppose first that $g \cdot x =y$. Then for $t$ with $\text{ht}(t)=j_{C_n}$ we have 
\begin{align*}
T_y(t) &= \{\ell : y \in U_{s_t}^{-1} \cdot W_{\ell} \} = \{\ell : g \cdot x \in U_{s_t}^{-1} \cdot W_{\ell} \} \\
 &= \{\ell :  x \in g^{-1}U_{s_t}^{-1} \cdot W_{\ell} \}  = \{\ell :  x \in (U_{s_t}g)^{-1} \cdot W_{\ell} \} \\
 &= \{\ell :  x \in U_{s_t^g}^{-1} \cdot W_{\ell} \} = T_x(\varphi(t)),
\end{align*}
 where $\varphi \in \Aut(T_{\Gamma})$ is induced by sending each node of the form $d^s_0 {}^\smallfrown \cdots {}^\smallfrown d^s_{n-1}$ to $g^{-1}(d^s_0) ^\smallfrown \cdots {}^\smallfrown  g^{-1}(d^s_{n-1})$. Then $\varphi$ is an isomorphism from $T_y$ to $T_x$.
 
Conversely, suppose $T_x \cong T_y$ via $\varphi \in \Aut(T_{\Gamma})$. Let $T'_x$ be the subtree of $T_x$ consisting of initial segments of those $t$ with $T_x(t) \neq \emptyset$ (equivalently, $U_{s_t} \neq \emptyset$), so that $[T'_x]$ is a Polish subspace of $\Gamma^{\omega}$.  Define the set of branches $A_x=\{\alpha \in [T'_x] : \text{$\bigcup_n s_{\alpha \upharpoonright j_{C_n}}$ is an automorphism of $T_{\Gamma}$} \}$. We claim that $A_x$ is comeager in $[T'_x]$. To see this, note that for $\alpha \in [T'_x]$ we will have that $\alpha \in A_x$ exactly when each node $r$ in $T_{\Gamma}$ appears as $d^s_i$ with $s=s_{\alpha \upharpoonright j_{C_n}}$ for some $i<n$, and for any such $r$ any node in $T'_x$ may be extended to another node in $T'_x$ for which this is true. Now, since $\varphi$ induces a homeomorphism from $[T'_x]$ to $[T'_y]$, there are $\alpha \in A_x$ and $\beta \in A_y$ with $\varphi(\alpha)=\beta$ (i.e., $\varphi(\alpha \upharpoonright n) = \beta \upharpoonright n$ for all $n$). Let $g_{\alpha}, g_{\beta} \in \Aut(T_{\Gamma})$ be the induced automorphisms, so $\bigcap_n U_{s_{\alpha \upharpoonright j_{C_n}}} =\{g_\alpha\}$ and $\bigcap_n U_{s_{\beta \upharpoonright j_{C_n}}} =\{g_\beta\}$. Since $\{\ell : x \in U^{-1}{s_{\alpha \upharpoonright j_{C_n}}} \cdot W_{\ell} \} = \{\ell : y \in U^{-1}{s_{\beta \upharpoonright j_{C_n}}} \cdot W_{\ell} \}$ for all $n$ we must then have $g_{\alpha} \cdot x = g_{\beta} \cdot y$.
\end{proof}

The group $\Aut(T_{\Gamma})$ corresponds to the $\Gamma$-jump in an analogous way to that of the group $S_{\infty}$ with respect to the Friedman--Stanley jump, so it is natural to ask if it satisfies similar properties. For instance, Friedman showed:

\begin{thm}[Friedman, Theorem 1.5 of \cite{Friedman}]
If $E$ is a Borel equivalence relation reducible to some $S_{\infty}$-action, then $E$ is reducible to $F_{\alpha}$ for some $\alpha < \omega_1$.
\end{thm}

Allison has noted that the analogous result holds for $\Aut(T_{\Gamma})$-actions:

\begin{cor}[Allison]
If $E$ is a Borel equivalence relation reducible to some  $\Aut(T_{\Gamma})$-action, then $E$ is reducible to $\J{\Gamma}{\alpha}$ for some $\alpha < \omega_1$.
\end{cor}

\begin{proof}
Since $\Aut(T_{\Gamma})$ is a closed subgroup of $S_{\infty}$, Friedman's theorem implies that $E$ is reducible to some $F_{\alpha}$, and hence to some $\bPi^0_{\alpha}$ orbit equivalence relation of $S_{\infty}$. By Theorem~2.4 of \cite{Allison}, $E$ is then reducible to an action of $\Aut(T_{\Gamma})$ with a potentially $\bPi^0_{\alpha}$ equivalence relation. By Theorem~2.7 of \cite{Allison}, $E$ is then reducible to a continuous action of $\Aut(T_{\Gamma})$ with a $\bPi^0_{\alpha}$ orbit equivalence relation, and hence to some $\J{\Gamma}{\alpha}$ by Theorem~\ref{thm:Gactions}.
\end{proof}

Friedman and Stanley asked in \cite{friedman-stanley} if every $E$ given by an $S_{\infty}$-action which satisfies $F_{\alpha} \leq_B E$ for all $\alpha < \omega_1$ must be Borel complete; this remains open. We can similarly ask:

\begin{question}
  If $E$ is given by an $\Aut(T_{\Gamma})$-action which satisfies $\J{\Gamma}{\alpha} \leq_B E$ for all $\alpha < \omega_1$, is $\cong^{\Gamma}$ reducible to $E$?
\end{question}

We note that there are Borel equivalence relations induced by actions of $\Aut(T_{\Gamma})$ which are not reducible to one induced by an action of $\Gamma^{\omega}$. This follows from the result of Allison and Panagiotopoulos that $E_0^{[\Z]}$ is generically ergodic  with respect to any orbit equivalence relation of a TSI Polish group (Corollary~2.3 of \cite{Allison-Aristotelis}) and the fact that $\Gamma^{\omega}$ is  TSI for every countable $\Gamma$.

We  do not know if there is a canonical obstruction to reducibility to $\Aut(T_{\Gamma})$-actions, in the way that turbulence is an obstruction for $S_{\infty}$-actions.

\begin{question}
Is there a dynamical characterization of when a Borel equivalence relation is reducible to an $\Aut(T_{\Gamma})$-action?
\end{question}

We close this section by noting that Shani has observed that $\Gamma$-jumps give ``natural'' examples of equivalence relations at intermediate levels of the Borel hierarchy, defined in Section 6 of \cite{hjorth-kechris-louveau}. Namely, the $\Z$-jump of $F_2$ (which is $\cong_2$ in the notation of \cite{hjorth-kechris-louveau}) is reducible to the relation $\cong^{\ast}_{3,0}$ defined there, and has potential complexity precisely $D(\bPi^0_3)$. We do not know for which other equivalence relations and groups this hollds, since the relations  $\cong^{\ast}_{\alpha,0}$ require pairwise-inequivalent successors at each node, and we do not know in general when $E^{[\Gamma]}$ is reducible to $E^{[\Gamma]}_{\text{p.i.}}$. We will see in Proposition~\ref{prop:sharp-bounds} that this same precise complexity is obtained for $E_0^{[\Z_2^{< \omega}]}$.

\section{Properness of the $\Gamma$-jump}
\label{sec:properness}

In this section we consider the question of when the $\Gamma$-jump is a proper jump. To begin, we use our results of the previous section together with a result of Solecki to establish that the $\Gamma$-jump is proper for a large class of countable groups, including the $\Z$-jump.

\begin{thm}
  \label{thm:proper}
Let $\Gamma$ be a countable group such that $\Z$ or $\Z_p^{<\omega}$ for a prime $p$ is a quotient of a subgroup of $\Gamma$. Then the $\Gamma$-jump is a proper jump operator.
\end{thm}

\begin{proof}
  Suppose towards a contradiction that $E$ is a Borel equivalence relation such that $\Delta(2) \leq_B E$ and $E^{[\Gamma]} \sim_B E$. Then by induction on $\alpha$ we have that $\J[E]{\Gamma}{\alpha} \leq_B E$ for all $\alpha<\omega_1$, successor stages following from our assumption and limit stages from Proposition~\ref{prop:Eomega} and the fact that $\J[E]{\Gamma}{\lambda} = \left(\bigoplus_{\alpha<\lambda}\J[E]{\Gamma}{\alpha}\right)^{[\Gamma]}  \leq_B \left(E^{\omega}\right)^{[\Gamma]}$ for a limit ordinal $\lambda$. Hence $\J{\Gamma}{\alpha} \leq_B E$ for all $\alpha$, so by Theorem~\ref{thm:Gactions}, every Borel orbit equivalence relation induced by an action of $\Aut(T_{\Gamma})$ is reducible to $E$. 
  It is a theorem of Solecki \cite[Theorem~1]{solecki} that if $\Delta$ is one of the groups $\Z$ or $\Z_p^{<\omega}$ for a prime $p$ then $\Delta^\omega$ admits non-Borel orbit equivalence relations, and by \cite{hjorth} it follows that $\Delta^{\omega}$ and hence $\Aut(T_{\Delta})$ induces Borel orbit equivalence relations of cofinal essential complexity. By our hypothesis and Proposition~\ref{prop:subquo} it then follows that $\Aut(T_{\Gamma})$ induces Borel orbit equivalence relations of cofinal essential complexity, contradicting that they are all reducible to $E$.
\end{proof}

Although the $\Gamma$-jump has no Borel fixed points for such $\Gamma$, there are always analytic equivalence relations $E$ with $E^{[\Gamma]} \mathrel{\sim_B} E$. Of course, the universal analytic equivalence relation is an example. Moreover, by Proposition~\ref{prop:iso-fixed}, the isomorphism relation $\cong^{\Gamma}$ on countable $\Gamma$-trees is a fixed point too. Thus we have the following:

\begin{cor}
  Let $\Gamma$ be a countable group such that $\Z$ or $\Z_p^{<\omega}$ for a prime $p$ is a quotient of a subgroup of $\Gamma$. Then $\cong^{\Gamma}$ is not Borel.
\end{cor}

For such $\Gamma$, $\cong^{\Gamma}$ gives a new example of a non-Borel, non-Borel complete isomorphism relation. Other known examples include for instance isomorphism of abelian $p$-groups \cite{friedman-stanley}, as well as several given in \cite{Ulrich-Rast-Laskowki}.

Friedman and Stanley's proof that $E <_B E^{+}$ utilized a theorem of Friedman's on the non-existence of Borel diagonalizers. We do not know if an analogous result holds for $\Gamma$-jumps.
 
 \begin{defn}
 A \emph{Borel diagonalizer} for $E^{[\Gamma]}$ is a Borel $E^{[\Gamma]}$-invariant mapping $f\colon X^\Gamma\to X$ such that for all $x \in X^{\Gamma}$ and $\gamma\in\Gamma$ we have $f(x)\not\mathrel{E}x(\gamma)$.
 \end{defn}
 
 The following is the analogue of Friedman--Stanley's application of diagonalizers to the jump:
 
\begin{lem}
  \label{lem:diagonalizer}
  If $E^{[\Gamma]} \leq_B E$ then there is a Borel diagonalizer for $\left(E^{[\Gamma]}\right)^{[\Gamma]}$.
\end{lem}
 
\begin{proof}
  Let $f$ be a Borel reduction from $E^{[\Gamma]}$ to $E$. We may find $z \in X^{\Gamma}$ so that $z_{\alpha} \mathrel{E} z_{\beta}$ for all $\alpha, \beta \in \Gamma$ and $f(z) \not\mathrel{E} z_{\alpha}$ for all $\alpha$. Define $F: \left(X^{\Gamma}\right)^{\Gamma} \rightarrow X^{\Gamma}$ by
  \[ F(\bar{x})(\gamma) = \begin{cases}
  f(x^{\gamma}), & \text{if $\forall \delta f(x^{\gamma}) \not\mathrel{E} x^{\gamma}_{\delta}$}\text{,} \\
  f(z), & \text{otherwise.}
  \end{cases}
  \]
  Then $F$ is $\left(E^{[\Gamma]}\right)^{[\Gamma]}$-invariant, and we claim that for all $\bar{x}$ and $\gamma$ we have $F(\bar{x}) \not\mathrel{E^{[\Gamma]}} x^{\gamma}$. Suppose instead that there is $\gamma_0$ with $F(\bar{x}) \mathrel{E^{[\Gamma]}} x^{\gamma_0}$. Let $A= \{x^{\alpha} : \forall \delta f(x^{\alpha}) \not\mathrel{E} x^{\alpha}_{\delta} \} \cup \{ z : \exists \gamma \exists \delta f(x^{\gamma}) \mathrel{E} x^{\gamma}_{\delta} \}$, so that $\ran(F(\bar{x})) = \{f(y) : y \in A\}$.
 
  If there is $y \in A$ with $f(x^{\gamma_0}) \mathrel{E} f(y)$, then $x^{\gamma_0} \mathrel{E^{[\Gamma]}} y$,  so $\forall \delta f(x^{\gamma_0}) \not\mathrel{E} x^{\gamma_0}_{\delta}$ (noting this is true for $z$ from its choice). Thus $x^{\gamma_0} \in A$, so $f(x^{\gamma_0}) \in \ran(F(\bar{x}))$. But $F(\bar{x}) \mathrel{E^{[\Gamma]}} x^{\gamma_0}$, so there would be $\delta$ with $f(x^{\gamma_0}) \mathrel{E} x^{\gamma_0}_{\delta}$, a contradiction. Hence there is no $y \in A$ with $f(x^{\gamma_0}) \mathrel{E} f(y)$, so there is no $\delta$ with $f(x^{\gamma_0}) \mathrel{E} x^{\gamma_0}_{\delta}$. But then we also have $x^{\gamma_0} \notin A$, and there is hence some $\delta$ with $f(x^{\gamma_0}) \mathrel{E} x^{\gamma_0}_{\delta}$, again producing a contradiction.
\end{proof}
  
\begin{question}
  Is it the case that $E^{[\Gamma]}$ doesn't admit a Borel diagonalizer when the $\Gamma$-jump is proper?  Note that this statement is at least as strong as Friedman's theorem, since a diagonalizer for $E^+$ easily produces a diagonalizer for $E^{[\Gamma]}$.
\end{question}

We now turn to the question of finding $\Gamma$-jumps that are not proper. Recall that a group satisfies the \emph{descending chain condition} if it does not have an infinite properly descending chain of subgroups. For example, the Pr\"{u}fer $p$-groups $\Z(p^{\infty})$ (also called quasi-cyclic groups) satisfy the descending chain condition. More generally if $\Gamma$ is \emph{quasi-finite}, meaning it is infinite with no infinite proper subgroups, then $\Gamma$ satisfies the descending chain condition.

We use the descending chain condition to obtain the condition that a descending chain of cosets of subgroups has nonempty intersection. One may verify that for countable groups, the descending chain condition is equivalent to this latter condition.

\begin{lem}
\label{lem:limitproduct}
If $\Gamma$ satisfies the descending chain condition, then $\J[E]{\Gamma}{\lambda}$ is Borel bireducible with $\prod_{\alpha<\lambda} \J[E]{\Gamma}{\alpha}$ for $\lambda$ a limit and $E$ with at least two classes.
\end{lem}

\begin{proof}
  First we show $\prod_{\alpha<\lambda} \J[E]{\Gamma}{\alpha} \leq_B \J[E]{\Gamma}{\lambda}$ for any infinite $\Gamma$. Let $\Gamma = \{\gamma_n : n \in \omega\}$ and $\lambda=\{\alpha_n : n \in \omega\}$. Define $f$ by $f(x)(\gamma_n) = x(\alpha_n)$; then $f$ is a reduction from $\prod_{\alpha<\lambda} \J[E]{\Gamma}{\alpha}$ to $\left(\bigoplus_{\alpha<\lambda}\J[E]{\Gamma}{\alpha}\right)^{[\Gamma]} = \J[E]{\Gamma}{\lambda}$.

  For the other direction, we define a reduction from $\J[E]{\Gamma}{\lambda} = \left(\bigoplus_{\alpha<\lambda}\J[E]{\Gamma}{\alpha}\right)^{[\Gamma]}$ to $\prod_{\alpha<\lambda} \left(\bigoplus_{\beta < \alpha} \J[E]{\Gamma}{\beta}\right)^{[\Gamma]}$, which is bireducible with $\prod_{\alpha<\lambda} \J[E]{\Gamma}{\alpha}$ when $E$ has at least two classes. Fix some $a$ in the domain of $E$ and let 
  \[f(x)(\alpha)(\gamma)= \begin{cases}
x(\gamma), & \text{if $x(\gamma)$ is in the domain of $\bigoplus_{\beta < \alpha} \J[E]{\Gamma}{\beta}$}\text{,} \\
a, &\text{otherwise.}
  \end{cases}
  \]
  Then $f$ is easily a homomorphism. Suppose that $f(x) \mathrel{\prod_{\alpha<\lambda} \left(\bigoplus_{\beta < \alpha} \J[E]{\Gamma}{\beta}\right)^{[\Gamma]}} f(x')$. For each $\alpha < \lambda$ let $H_{\alpha} = \left\{ \gamma : \gamma \cdot f(x)(\alpha) \mathrel{\left(\bigoplus_{\beta < \alpha} \J[E]{\Gamma}{\beta}\right)^{\Gamma}} f(x')(\alpha) \right\}$, so that the $H_{\alpha}$'s form a nonempty descending chain of cosets of subgroups of $\Gamma$. Since $\Gamma$ satisfies the descending chain condition, there is some $\gamma \in \bigcap_{\alpha<\lambda} H_{\alpha}$, and this $\gamma$ satisfies $\gamma \cdot x \mathrel{\left(\bigoplus_{\alpha<\lambda}\J[E]{\Gamma}{\alpha}\right)^{\Gamma}} x'$, so $x \mathrel{\J[E]{\Gamma}{\lambda}} x'$.
\end{proof}

Using Lemma~\ref{lem:freepiz}, one can show that the above result also holds when $\Gamma=\Z$, but we do not know whether it holds for other groups $\Gamma$.

\begin{thm}
  \label{thm:improper}
  If $\Gamma$ satisfies the descending chain condition, then $\J{\Gamma}{\omega+1} \sim_B \J{\Gamma}{\omega}$. In particular, for such $\Gamma$ the $\Gamma$-jump is not a proper jump operator.
\end{thm}

\begin{proof}
  From the previous lemma we have $\J{\Gamma}{\omega+1} \sim_B \left(\prod_{n \in \omega} \J{\Gamma}{n} \right)^{[\Gamma]}$  and $\J{\Gamma}{\omega} \sim_B \prod_{n \in \omega} \left(\prod_{k \in n} \J{\Gamma}{k} \right)^{[\Gamma]}$ so it suffices to define a Borel reduction $\varphi$ from $\left(\prod_{n \in \omega} \J{\Gamma}{n} \right)^{[\Gamma]}$ to $\prod_{n \in \omega} \left(\prod_{k \in n} \J{\Gamma}{k} \right)^{[\Gamma]}$. For this we let $\varphi(x)(n)(\gamma) = x(\gamma) \upharpoonright n$.
  
  To see that $\varphi$ is a homomorphism, we calculate:
  \begin{align*}
    x \mathrel{ \left(\prod_{n \in \omega} \J{\Gamma}{n} \right)^{[\Gamma]}} y 
    &\iff \exists \gamma \forall \alpha \forall n\; x (\gamma^{-1} \alpha)(n) \mathrel{\J{\Gamma}{n}} y(\alpha)(n) \\ 
    &\Longrightarrow \forall n \exists \gamma \forall \alpha \forall k \in n \; x (\gamma^{-1} \alpha)(k) \mathrel{\J{\Gamma}{k}} y(\alpha)(k) \\
    & \iff \forall n \exists \gamma \forall \alpha \; \varphi(x)(n)(\gamma^{-1} \alpha)  \mathrel{\left(\prod_{k \in n} \J{\Gamma}{k} \right)} \varphi(y)(n)(\alpha) \\
    & \iff \varphi(x) \mathrel{\prod_{n \in \omega} \left(\prod_{k \in n} \J{\Gamma}{k} \right)^{[\Gamma]}} \varphi(y) .
  \end{align*}
  
  Now suppose conversely that $\varphi(x) \mathrel{\prod_{n \in \omega} \left(\prod_{k \in n} \J{\Gamma}{k} \right)^{[\Gamma]}} \varphi(y)$. For each $n$, let 
  \[ H_n = \left\{ \gamma : \forall \alpha \; x(\gamma^{-1} \alpha) \upharpoonright n \mathrel{\left(\prod_{k \in n} \J{\Gamma}{k} \right)} y(\alpha) \upharpoonright n \right\}.
  \]
  Then the sequence $H_n$ is a descending chain of cosets of subgroups of $\Gamma$. It follows from the descending chain condition that $\bigcap_n H_n \neq \emptyset$. If $\gamma \in \bigcap_n H_n$, then $\gamma$ witnesses that $x \mathrel{ \left(\prod_{n \in \omega} \J{\Gamma}{n} \right)^{[\Gamma]}} y$, completing the proof.
\end{proof}

We do not know if the statement of Theorem~\ref{thm:improper} is tight in the sense that $\J{\Gamma}{\omega}$ is the least fixed-point. If $\Gamma$ is a group such that the $\Gamma$-jump is improper, one may ask what is the least $\alpha$ such that $\J{\Gamma}{\alpha} = \J{\Gamma}{\alpha+1}$. In the next section we will show that for all $\Gamma$ we have $\J{\Gamma}{0} <_B \J{\Gamma}{1} <_B \J{\Gamma}{2}$.

The following is a consequence of Theorem~\ref{thm:improper} together with the proofs of Theorem~\ref{thm:proper} and Corollary~\ref{cor:gamma-omega-actions}.

\begin{cor}
  If $\Gamma$ satisfies the descending chain condition, then the family of Borel orbit equivalence relations induced by continuous actions of the group $\Aut(T_\Gamma)$ is bounded with respect to $\leq_B$. 
\end{cor}

For abelian groups $G$, having non-Borel orbit equivalence relations is equivalent to the family of Borel orbit equivalence relations induced by continuous actions of $G$ being unbounded with respect to $\leq_B$ by Theorem~5.11 of \cite{hjorth}, but we do not know if this holds for $\Aut(T_{\Gamma})$. It suffices to ask about $\cong^{\Gamma}$:

\begin{question}
If $\Gamma$ satisfies the descending chain condition, is $\cong^{\Gamma}$ Borel?
\end{question}

From Lemma~\ref{lem:diagonalizer} we also have:
\begin{cor}
 If $\Gamma$ satisfies the descending chain condition then there is a Borel diagonalizer for $\J{\Gamma}{\omega+2}$.
\end{cor}

We conclude this section by exploring the gap between our results on proper and improper $\Gamma$-jumps. In light of Theorem~\ref{thm:proper}, it is natural to ask for which countable groups $\Gamma$ we have non-Borel orbit equivalence relations of $\Gamma^\omega$ or of $\Aut(T_{\Gamma})$. For abelian groups $\Gamma$, the answer to this question is known in the case of $\Gamma^\omega$-actions. We recall the following definition from \cite{solecki}.

\begin{defn}
  Let $\Gamma$ be a group and $p$ a prime number. Then $\Gamma$ is said to be \emph{$p$-compact} if, for any descending chain of subgroups $G_k\leq\Z/p\Z \times \Gamma$ such that $(\forall k)\,\pi[G_k] = \Z/p\Z$, we have $\pi[ \bigcap_{k \in \omega} G_k] = \Z/p\Z$. (Here $\pi \colon \Z/p\Z \times \Gamma \rightarrow \Z/p\Z$ denotes the projection.)
\end{defn}

If $\Gamma$ satisfies the descending chain condition, then $\Gamma$ is $p$-compact for all primes $p$. Indeed, given $G_k$ as above, for each $a \in \Z_p$ let $H^a_k = \{g \in \Gamma : (a,g) \in G_k \}$. Then $H^a_k$ is a descending chain of cosets of subgroups of $\Gamma$, so their intersection is non-empty, meaning $a \in \pi[ \bigcap_{k \in \omega} G_k]$.

An example of a group which is $p$-compact for all $p$, but does not satisfy the descending chain condition, is $\Gamma=\bigoplus_{\text{$p$ prime}} \Z_p$. Moreover $\Gamma$ is a group for which the hypotheses of both Theorem~\ref{thm:proper} and Theorem~\ref{thm:improper} are not satisfied.

It is natural to ask whether Theorem~\ref{thm:improper} can be generalized to all groups which are $p$-compact for all $p$. Indeed, it is shown in \cite[Theorem~2]{solecki} that if $\Gamma$ is $p$-compact for all $p$, then $\Gamma^{\omega}$ does not have a non-Borel orbit equivalence relation. It is further shown in \cite{solecki} that for abelian groups $\Gamma$, this is a complete characterization. The question of which non-abelian $\Gamma$ admit non-Borel $\Gamma^{\omega}$ orbit equivalence relations is open. Of course it may also be possible that some $\Gamma$-jump is proper without this condition holding. 

\begin{question}
  Which countable groups $\Gamma$ give rise to proper jump operators?  If $\Gamma$ is $p$-compact for all primes $p$, is the $\Gamma$-jump improper? Does the group $\bigoplus_{\text{$p$ prime}} \Z_p$ give rise to a proper jump operator?
\end{question}

\section{Bounds on potential complexities}
\label{sec:bounds}

The bounds in the statement of Theorem~\ref{thm:Gactions} are not always tight, that is, sometimes it is possible to reduce an equivalence relation $E$ to fewer iterates of the jump. As a consequence, the lower bounds on the potential complexity of $\J{\Gamma}{\alpha}$ that it gives are also not always tight.
The bounds provided in Theorem~\ref{thm:Gactions} are derived from the definitions of the equivalence relations introduced in the proof, together with Proposition~\ref{prop:Eomega}, Lemma~\ref{lem:Zk}, and the fact that $\prod_{\alpha<\lambda} \J{\Gamma}{\alpha} \leq_B \J{\Gamma}{\lambda}$ for limit $\lambda$ (see the proof of Lemma~\ref{lem:limitproduct}).
In general this technique requires about $\omega$-many iterates of the $\Gamma$-jump to ensure the potential complexity has increased by one level in the Borel hierarchy.

In this section we provide a more direct proof that the iterated jumps have cofinal potential complexity in the special case when $\Gamma=\Z_2^{<\omega}$, and produce sharp bounds on potential complexity of the iterates of the $\Z_2^{< \omega}$-jump. As a consequence we obtain an alternate proof that the $\Z_2^{< \omega}$-jump is proper. Recall that for a pointclass $\Gamma$, the pointclass $D(\Gamma)$ consists of all sets which are the difference of two sets in $\Gamma$, and the pointclass $\check{D}(\Gamma)$ is its dual, consisting of the complements of all such sets.

\begin{prop}
  \label{prop:sharp-bounds}
  For $\Gamma=\Z_2^{< \omega}$ we have the following:
  \begin{itemize}
    \item $\J{\Z_2^{< \omega}}{\alpha} \in \pot( D(\bPi^0_{\alpha+1})) \setminus \pot( \check{D}(\bPi^0_{\alpha+1}))$ for $\alpha \geq 2$ not a limit;
    \item $\J{\Z_2^{< \omega}}{\lambda} \in \pot(\bSigma^0_{\lambda+1}) \setminus \pot(\bPi^0_{\lambda+1})$ for $\lambda$ a limit.
  \end{itemize}
\end{prop}

\begin{proof}
  The upper bounds follow from Corollary~\ref{cor:upperbounds}. To establish the lower bounds, we use the properly increasing tower $A_\alpha$ of equivalence relations defined by Hjorth--Kechris--Louveau in \S5 of \cite{hjorth-kechris-louveau} (where our $A_n$ corresponds to their  $E_{G_{n+1}}$ for finite $n$, and $A_{\alpha}$ to $E_{G_{\alpha}}$ for infinite $\alpha$), and show that $A_{\alpha} \leq_B \J{\Gamma}{\alpha}$ for each $\alpha$. We will let $\Gamma$ denote $\Z_2^{< \omega}$ throughout this proof. To define $A_\alpha$, we adopt the following notation. For any $x\in 2^\omega$ we may regard $x$ as an element of $(2^\omega)^\omega$, and we write $x_n$ for the $n$th element of $2^\omega$ in $x$. Also for any $x\in 2^\omega$ we let $\bar x(n)=1-x(n)$.
  
  We let $A_0=\Delta(2)$ and $A_1=E_0$. Given $A_{\alpha}$, we define $x\mathrel{A}_{\alpha+1}y$ if and only if:
  \begin{enumerate}
    \item for all $n$, either $x_n\mathrel{A}_\alpha y_n$ or $\overline{x_n}\mathrel{A}_\alpha y_n$, and
    \item for all but finitely many $n$, $x_n\mathrel{A}_\alpha y_n$.
  \end{enumerate}
 For a limit ordinal $\lambda$ we fix an increasing sequence $\langle \alpha_n \rangle_{n \in \omega}$ cofinal in $\lambda$ and define $x\mathrel{A}_{\lambda}y$ if and only if:
  \begin{enumerate}
    \item for all $n$, either $x_n\mathrel{A}_{\alpha_n} y_n$ or $\overline{x_n}\mathrel{A}_{\alpha_n} y_n$, and
    \item for all but finitely many $n$, $x_n\mathrel{A}_{\alpha_n} y_n$.
  \end{enumerate}

 Hjorth--Kechris--Louveau show, in Theorem~5.8 of \cite{hjorth-kechris-louveau} and its proof, that (with our indexing)  $A_{\alpha} \notin \pot( \check{D}(\bPi^0_{\alpha+1}))$ for $\alpha \geq 2$ not a limit, and that $A_{\lambda} \notin \pot(\bPi^0_{\lambda+1})$ for $\lambda$ a limit. Hence the desired lower bounds will be established once we show that each $A_{\alpha} \leq_B \J{\Gamma}{\alpha}$.

  for $\alpha=0$ and $\alpha=1$, this is immediate. Next suppose $A_{\alpha} \leq_B \J{\Gamma}{\alpha}$; we will show that $A_{\alpha+1} \leq_B \J{\Gamma}{\alpha+1}$.
We first claim that $A_{\alpha+1}\leq_B((A_\alpha)^\omega)^{[\Gamma]}$ for each $\alpha$. Admitting this claim, since $\Z_2^{< \omega} \cong \Z_2^{< \omega} \times \Z_2^{< \omega}$, we have $\left(E^{\omega}\right)^{[\Gamma]} \leq_B E^{[\Gamma]}$ for any $E$ by Proposition~\ref{prop:Eomegajump}. Hence if $A_{\alpha} \leq_B \J{\Gamma}{\alpha}$ then $A_{\alpha+1} \leq_B \left((\J{\Gamma}{\alpha})^\omega\right)^{[\Gamma]}\leq_B  \left(\J{\Gamma}{\alpha}\right)^{[\Gamma]} = \J{\Gamma}{\alpha+1}$.

  To establish the claim, let $x\in2^\omega$ be given and define $\alpha_x(n,s)$ for $n\in\omega$ and $s\in2^{<\omega}$ by
  \[\alpha_x(n,s)=\begin{cases}x_n,&\text{if $s(n)=0$,}\\\overline{x_n},&\text{if $s(n)=1$.}\end{cases}
  \]
  Then if $x\mathrel{A}_{\alpha+1}y$, we define $\gamma\in\Gamma$ by $\gamma(n)=0$ if $x_n\mathrel{A}_\alpha y_n$ and $\gamma(n)=1$ otherwise. Then $\gamma$ witnesses that $\alpha_x$ is equivalent to $\alpha_y$.
  On the other hand, if $\alpha_x$ is equivalent to $\alpha_y$, let $\gamma\in\Gamma$ witness this. Then the coordinates of $\gamma$ witness that $x\mathrel{A}_{\alpha+1}y$.
  
For a limit ordinal $\lambda$, suppose $A_{\alpha} \leq_B \J{\Gamma}{\alpha}$ for all $\alpha<\lambda$; we show that $A_{\lambda} \leq_B \J{\Gamma}{\lambda}$. A direct modification of the proof of the previous claim gives that $A_{\lambda} \leq_B ( \prod_n A_{\alpha_n})^{[\Gamma]}$. Similarly, a direct modification of the proof of Proposition~\ref{prop:Eomegajump} gives that $( \prod_n A_{\alpha_n})^{[\Gamma]} \leq_B ( \bigoplus_n A_{\alpha_n})^{[\Gamma]}$ for $\Gamma=\Z_2^{< \omega}$. Since 
\[\textstyle ( \bigoplus_n A_{\alpha_n})^{[\Gamma]} \leq_B ( \bigoplus_{\alpha < \lambda} A_{\alpha})^{[\Gamma]} \leq_B ( \bigoplus_{\alpha < \lambda} \J{\Gamma}{\alpha})^{[\Gamma]} = \J{\Gamma}{\lambda}, \]
this completes the proof.
\end{proof}

We do not know the optimal complexity bounds in the case of groups other than $\Z_2^{< \omega}$. In particular we do not know precise bounds on the complexities of the iterates of the $\Z$-jump.

\begin{question}
  What are the exact potential Borel complexities of the $Z_{\alpha}$'s?
\end{question}

\section{Generic $E_0$-ergodicity of $\Gamma$-jumps}
\label{sec:generic}

In this section we will show that $\Gamma$-jumps of countable Borel equivalence relations produce new examples of equivalence relations intermediate between $E_0^{\omega}$ and $F_2$. Furthermore, we will see that although some $\Gamma$-jumps admit Borel fixed points, any such fixed point must be strictly above $E_0$. The key notion in establishing these results is the following.

\begin{defn}
We say that $E$ is \emph{generically $F$-ergodic} if for every Baire measurable homomorphism $\varphi$ from $E$ to $F$ there is a $y$ so that $\varphi^{-1}[y]_F$ is comeager. We say $E$ is \emph{generically ergodic} when it is generically $\Delta(\R)$-ergodic.
\end{defn}

Note that if $E$ is generically $F$-ergodic, and $E$ does not have a comeager equivalence class, then $E \not\leq_B F$. Furthermore observe that each $E_0^{[\Gamma]}$-class is meager. Also, when $E$ is an orbit equivalence relation, the following lemma shows that we may assume $\varphi$ is Borel when verifying generic $F$-ergodicity. 

\begin{lem}
Let $E = E_G^X$ be an orbit equivalence relation on $X$ and $F$ an equivalence relation on $Y$. If $\varphi: X \rightarrow Y$ is a Baire measurable homomorphism from $E$ to $F$, then there is a Borel homomorphism $\tilde{\varphi}$ from $E$ to $F$ with $\tilde{\varphi}(x) \mathrel{F} \varphi(x)$ for a comeager set of $x$.
\end{lem}

\begin{proof}
Let $C \subseteq X$ be a dense $G_{\delta}$ set so that $\varphi$ is continuous on $C$. Let $C^{\ast} = \{x: (\forall^{\ast} g)\,(g \cdot x \in C)\}$ be the Vaught transform, so that $C^{\ast}$ is an $E$-invariant dense $G_{\delta}$ set (see Lemma~5.1.7  of \cite{becker-kechris}). Let $P \subseteq C^{\ast} \times G$ be given by:
\[ P = \{ (x,g) : g \cdot x \in C\}, \]
so that each section $P_x$ is comeager. By Category Uniformization (see Corollary~18.7 of \cite{kechris-classical}) there is a Borel uniformization $\tilde{P} \subseteq P$ which is the graph of a Borel function $s : C^{\ast} \rightarrow G$. Fix an arbitrary element $y_0 \in Y$; we now define $\tilde{\varphi}$ by:
\[ \tilde{\varphi}(x) = \begin{cases}
\varphi( s(x) \cdot x) & \text{if $x \in C^{\ast}$,} \\
y_0 & \text{if $x \notin C^{\ast}$.}
\end{cases} \]
Then $\tilde{\varphi}$ is as desired. 
\end{proof}

Our main result will be to show that $E_0^{[\Gamma]}$ is generically $E_{\infty}$-ergodic for any countable group $\Gamma$. Throughout this section we let $\Gamma=\{\gamma_n : n \in \omega\}$ be a countably infinite group and fix an increasing sequence of finite subsets $\Gamma_n$ with $\Gamma= \bigcup_n \Gamma_n$.

\begin{defn}
Let $h$ be a Borel $\Z$-action inducing $E_0$, and let $G = \Z^{\Gamma}$ act coordinate-wise via $h$ to induce the orbit equivalence relation $E_0^{\Gamma}$. Then $G$ together with the shift action of $\Gamma$ on $\left(2^{\omega}\right)^{\Gamma}$ generates $E_0^{[\Gamma]}$. For each $n$ and each $\bar{k} \in \Z^{\Gamma_n}$, let $U^n_{\bar{k}}=\{ g \in G : g \upharpoonright \Gamma_n= \bar{k}\}$, so that $\{U^n_{\bar{k}} : n,\bar{k} \}$ is a countable basis for $G$.
\end{defn}

The following lemma is derived from Theorem~7.3 of \cite{hjorth-kechris}:

\begin{lem}
  \label{lem:homfinite}
  Let $f\colon E_0^{[\Gamma]} \rightarrow E_{\infty}$ be a Borel homomorphism. Then there is a Borel $\tilde{f}$ such that $\tilde{f}(x) \mathrel{E_{\infty}} f(x)$ for all $x$, and there are $E_0^{\Gamma}$-invariant sets $X_n$ with $\left(2^{\omega}\right)^{\Gamma} = \amalg_n X_n$ such that for $x, x' \in X_n$ with $x \mathrel{E_0^{\Gamma}} x'$ and $x \upharpoonright \Gamma_n = x' \upharpoonright \Gamma_n$ we have $\tilde{f}(x)=\tilde{f}(x')$.
\end{lem}

\begin{proof}
For each $x$, $[f(x)]_{E_{\infty}} = \{y_i : i \in \omega\}$ is countable, so there is some $i$ so that $G_i = \{ g \in G : f(g \cdot x) = y_i \}$ is nonmeager, and hence comeager in some $U^n_{\bar{k}}$. Hence:
\[ \forall x \exists n \exists \bar{k} \exists y \in [f(x)]_{E_{\infty}} \forall^{\ast} g \in U^n_{\bar{k}} \ (f(g \cdot x ) = y). \]
Set $P(x,n) \Leftrightarrow \exists \bar{k} \exists y \in [f(x)]_{E_{\infty}} \forall^{\ast} g \in U^n_{\bar{k}} \ (f(g \cdot x ) = y)$, so $\forall x \exists n P(x,n)$. Since $P$ is Borel and $E_0^{\Gamma}$-invariant (if $P(x,n)$ and $x \mathrel{E_0^{\Gamma}} x'$ then $P(x',n)$), there is a $E_0^{\Gamma}$-invariant  Borel function $s: \left(2^{\omega}\right)^{\Gamma} \rightarrow \N$ with $P(x,s(x))$ for all $x$. Let $X_n = s^{-1}[\{n\}]$.

For $x \in X_n$, let $\bar{k}(x)$ be the least $\bar{k}$ (in some fixed enumeration of $\Z^{\Gamma_n}$) so that $\exists y \in [f(x)]_{E_{\infty}} \forall^{\ast} g \in U^n_{\bar{k}} \ (f(g \cdot x ) = y)$. Note that if $x \mathrel{E_0^{\Gamma}} x'$ and $x \upharpoonright \Gamma_n = x' \upharpoonright \Gamma_n$ then $\bar{k}(x)=\bar{k}(x')$. We can now let $\tilde{f}(x)$ be the unique $y$ so that $\forall^{\ast} g \in U^n_{\bar{k}(x)} \ (f(g \cdot x ) = y)$.
\end{proof}

Now $X_n$ is nonmeager for some $n$, and since $E_0^{\Gamma}$ is induced by a continuous group action with dense orbits it must be comeager. Let $n_0$ be this $n$. Let $Y_0$ be a comeager set with $\tilde{f}$ continuous on $Y_0$. Let $Y =  \bigcap_{n \in \omega} \gamma_n[ Y_0^{\ast G} \cap X_{n_0}]$ so that $Y$ is comeager and $E_0^{[\Gamma]}$-invariant.

\begin{lem} 
For all $x, x' \in Y$ with $x \upharpoonright \Gamma_{n_0}=x' \upharpoonright \Gamma_{n_0}$ we have $\tilde{f}(x) = \tilde{f}(x')$.
\end{lem}

\begin{proof}
Let $x, x' \in Y$ with $x \upharpoonright \Gamma_{n_0}=x' \upharpoonright \Gamma_{n_0}$. We can then choose $g\in G$ with $g \cdot x, g \cdot x' \in Y_0$ and $g \cdot x \upharpoonright \Gamma_{n_0} = g \cdot x' \upharpoonright \Gamma_{n_0} = x \upharpoonright \Gamma_{n_0}$. Hence $\tilde{f}(g \cdot x) = \tilde{f}(x)$ and $\tilde{f}(g \cdot x') = \tilde{f}(x')$, so, replacing $x$ and $x'$ by $g \cdot x$ and $g \cdot x'$, we may assume $x,x' \in Y_0$.

Suppose $\tilde{f}(x) \neq \tilde{f}(x')$, and let $V$ and $V'$ be disjoint neighborhoods with $\tilde{f}(x) \in V$ and $\tilde{f}(x') \in V'$. We can then find neighborhoods $U$ and $U'$ in $\left(2^{\omega}\right)^{\Gamma}$ with $x \in U$ and $x' \in U'$, so that if $y \in U \cap Y_0$ then $\tilde{f}(y) \in V$ and if $y' \in U' \cap Y_0$ then $\tilde{f}(y') \in V'$. Since $G$-orbits are dense, the set $\{ g \in G: g \cdot x \in U' \wedge g \cdot x \upharpoonright \Gamma_{n_0} = x \upharpoonright \Gamma_{n_0} \}$ is non-empty and open, so it contains some $g$ with $g \cdot x \in Y_0$. But then $\tilde{f}(g \cdot x) = \tilde{f}(x)$ by the previous lemma; however, then $\tilde{f}(g \cdot x) \in V'$ but $\tilde{f}(x) \in V$, a contradiction.
\end{proof}

\begin{lem}
For all $x, x' \in Y$, if there is $\gamma \in \Gamma$ with $x \upharpoonright \gamma \Gamma_{n_0} = x' \upharpoonright \gamma \Gamma_{n_0}$ then $\tilde{f}(x) \mathrel{E_{\infty}} \tilde{f}(x')$.
\end{lem}

\begin{proof}
  We have $\gamma^{-1} \cdot x \upharpoonright \Gamma_{n_0} = \gamma^{-1} \cdot x' \upharpoonright \Gamma_{n_0}$, and since $Y$ is $E_0^{[\Gamma]}$-invariant we have $\tilde{f}(\gamma^{-1} \cdot x)=\tilde{f}(\gamma^{-1} \cdot x')$. Since $x \mathrel{E_0^{[\Gamma]}} \gamma^{-1} \cdot x$ and $\tilde{f}$ is a homomorphism we have $\tilde{f}(x) \mathrel{E_{\infty}} \tilde{f}(\gamma^{-1} \cdot x)$, and similarly for $x'$, so the result follows.
\end{proof}

Now we are ready for the main result of this section:

\begin{thm}
  $E_0^{[\Gamma]}$ is generically $E_{\infty}$-ergodic.
\end{thm}

\begin{proof}
Let $f\colon E_0^{[\Gamma]} \rightarrow E_{\infty}$ be a Borel homomorphism, and let $\tilde{f}$, $n_0$, and $Y$ be as above. By the Kuratowski--Ulam theorem we have 
\[ \forall^{\ast} z \in \left(2^{\omega}\right)^{\Gamma_{n_0}} \forall^{\ast} w \in \left(2^{\omega}\right)^{\Gamma \setminus \Gamma_{n_0}} (z ^\smallfrown w \in Y), \]
so the set
\[Y' = \{x \in Y: \forall^{\ast} w \in \left(2^{\omega}\right)^{\Gamma \setminus \Gamma_{n_0}} ( x \upharpoonright \Gamma_{n_0} {}^\smallfrown w \in Y) \}\]
is comeager.
 It suffices to show that $\tilde{f}$ maps $Y'$ into a single $E_{\infty}$-class. Fix $x, x' \in Y'$, and let $\gamma \in \Gamma$ so that $\Gamma_{n_0}$ and $\gamma\Gamma_{n_0}$ are disjoint. We can then find $w \in \left(2^{\omega}\right)^{\Gamma \setminus \Gamma_{n_0}}$ so that $y=x \upharpoonright \Gamma_{n_0} {}^\smallfrown w \in Y$ and $y'=x' \upharpoonright \Gamma_{n_0} {}^\smallfrown w \in Y$. We have $\tilde{f}(x) \mathrel{E_{\infty}} \tilde{f}(y)$ and $\tilde{f}(x') \mathrel{E_{\infty}} \tilde{f}(y')$ by agreement on $\Gamma_{n_0}$, and $\tilde{f}(y) \mathrel{E_{\infty}} \tilde{f}(y')$ by agreement on $\gamma\Gamma_{n_0}$, so $\tilde{f}(x) \mathrel{E_{\infty}} \tilde{f}(x')$.
\end{proof}

When $E$ is generically $F$-ergodic it is also generically $F^{\omega}$-ergodic, so we immediately get:

\begin{thm}
$E_0^{[\Gamma]}$ is generically $E_{\infty}^{\omega}$-ergodic.\qed
\end{thm}

\begin{cor}
$E_0^{[\Z]}$ is generically $E_0^{\omega}$-ergodic.\qed
\end{cor}

Since $(E_0)^\omega$ is Borel reducible to $E_0^{[\Z]}$, we have in particular that $(E_0)^\omega$ is properly below $E_0^{[\Z]}$ in complexity.

Allison and Panagiotopoulos have since strengthened the last result to show that $E_0^{[\Z]}$ is generically ergodic with respect to any orbit equivalence relation of a TSI Polish group (Corollary~2.3 of \cite{Allison-Aristotelis}). More recently, Allison has shown that an equivalence relation $E$ is generically ergodic with respect to any orbit equivalence relation $E_G^X$ with $G$ non-Archimedean abelian if and only if $E$ is generically $E_0$-ergodic (Theorem~6.6 of \cite{Allison}), and $E$ is generically ergodic with respect to any orbit equivalence relation $E_G^X$ with $G$ non-Archimedean TSI if and only if $E$ is generically $E_{\infty}$-ergodic (Theorem~6.5 of \cite{Allison}).

From the above, we can see that $\Gamma$-jumps do not have fixed points at the first two levels.

\begin{cor}
For any countable group $\Gamma$ we have $\J{\Gamma}{0} <_B \J{\Gamma}{1} <_B \J{\Gamma}{2}$.
\end{cor}

\begin{proof}
For any countable $\Gamma$, $\J{\Gamma}{0} = \Delta(2)$, and $\J{\Gamma}{1}$ is the shift action of $\Gamma$ on $2^{\Gamma}$, which is countable and generically ergodic.
\end{proof}

In Section~\ref{sec:scattered} below we consider which countable Borel equivalence relations are reducible to $E_0^{[\Z]}$, and establish in Corollary~\ref{cor:Einfty-not-reducible} that $E_{\infty} \not\leq_B E_0^{[\Z]}$, which immediately implies that $E_{\infty}^{\omega} \not\leq_B E_0^{[\Z]}$.
We may now summarize the situation to see that $E_0^{[\Z]}$ and $E_{\infty}^{[\Z]}$ are new examples of natural equivalence relations below $F_2$ which provide genuinely new levels in the Borel complexity hierarchy. 

\begin{thm}
  We have the following:
  \begin{itemize}
    \item $E_0^{\omega} <_B E_0^{[\Z]} <_B E_{\infty}^{[\Z]} <_B F_2$;
    \item $E_0^{\omega} <_B E_{\infty}^{\omega} <_B E_{\infty}^{[\Z]} <_B F_2$;
    \item $E_0^{[\Z]}$ and $E_{\infty}^{\omega}$ are $\leq_B$-incomparable.\qed
  \end{itemize}
\end{thm}

Previously, the only known examples of equivalence relations between $E_{\infty}^{\omega}$ and $F_2$ were the examples constructed in \cite{zap-pinned}, which are non-pinned. Shani has also produced a new example of an intermediate pinned equivalence relation in \cite{shani}, denoted $E_{\Pi}$, which is also incomparable to $E_0^{[\Z]}$ and $E_{\infty}^{[\Z]}$.

\begin{thm}[Shani, Theorem 1.7 of \cite{shani}]
  $E_{\Pi}$ is pinned and satisfies:
  \begin{itemize}
    \item $E_{\infty}^{\omega} <_B E_{\Pi} <_B F_2$;
    \item $E_{\Pi} \not\leq_B E_{\infty}^{[\Gamma]}$ and $E_0^{[\Gamma]} \not\leq_B E_{\Pi}$ for any infinite countable group $\Gamma$.
  \end{itemize}
\end{thm}

Note that these results give several other equivalence relations strictly between $E_{\infty}^{\omega}$ and $F_2$, such as $E_{\infty}^{\omega} \times E_0^{[\Z]}$ and $E_{\infty}^{[\F_2]}$ (where $\F_2$ is the free group on 2 generators). Shani has shown, as consequences of Corollary 5.6 and Proposition 5.14 of \cite{shani}:

\begin{thm}[Shani]
  The following hold:
  \begin{itemize}
    \item $E_{\infty}^{[\Z]} \not\leq_B E_0^{[\Z]} \times E_{\infty}^{\omega}$;
    \item $E_0^{[\Z]} \times E_0^{\omega} \not\leq_B E_{\infty}^{[\Z]}$;
    \item  $E_{\infty}^{[\Z]}$ and $E_0^{[\Z]} \times E_{\infty}^{\omega}$ are $\leq_B$-incomparable;
    \item $E^{[\Z]} <_B \left(E^{[\Z]}\right)^2$ for any generically ergodic countable Borel equivalence relation $E$.
  \end{itemize}
\end{thm}

One can also ask if there are any equivalence relations strictly between $E_0^{\omega}$ and $E_0^{[\Z]}$.

\begin{question}
  If $E \leq_B E_0^{[\Z]}$, does either $E \sim_B E_0^{[\Z]}$ or $E \leq_B E_0^{\omega}$?
\end{question}

From the Hjorth--Kechris dichotomy for $E_0^{\omega}$, we can obtain a weak partial result.

\begin{lem}
If $E \leq_B E_0^{[\Z]}$ then either $E_0^{\omega} \leq_B E$ or $E \leq_B E_{\infty}$.
\end{lem}

\begin{proof}
Let $f$ be a reduction of $E$ to $E_0^{[\Z]}$. Set $x \mathrel{F} x'$ iff $f(x) \mathrel{E_0^{\Z}} f(x')$, so that $F$ is of countable index below $E$ and $F \leq_B E_0^{\Z} \sim_B E_0^{\omega}$. Hence either $E_0^{\omega} \leq_B F$ or $F \leq_B E_0$. In the first case, let $g$ be a reduction of $E_0^{\omega}$ to $F$, and set $x \mathrel{E'} x'$ iff $g(x) \mathrel{E} g(x')$. Then $E' \leq_B E$ and  $E'$ is of countable index over $E_0^{\omega}$, so that $E_0^{\omega} \leq_B E'$ (see Corollary 8.32 of \cite{clemens}) and hence $E_0^{\omega} \leq_B E$. In the second case, let $h$ be a reduction of $F$ to $E_0$, and define the equivalence relation $E'$ by $y \mathrel{E'} y'$ iff $y \mathrel{E_0} y' \vee \exists x, x' (x \mathrel{E} x' \wedge h(x) \mathrel{E_0} y \wedge h(x') \mathrel{E_0} y')$. Then $E'$ is a countable analytic equivalence relation with $E \leq_B E'$, so $E$ is essentially countable (see, e.g., Proposition 7.1 of \cite{clemens-lecomte-miller}); hence $E \leq_B E_{\infty}$.
\end{proof}

\begin{question}
If $E \leq_B E_{\infty}^{[\Z]}$, does either $E_0^{[\Z]} \leq_B E$ or $E \leq_B E_{\infty}^{\omega}$?
\end{question}

\section{$\Z$-jumps and scattered linear orders}
\label{sec:scattered}

In this section we give an application of our results about the $\Z$-jump to the classification of countable scattered linear orders. We also consider countable Borel equivalence relations which admit an assignment of scattered linear orders to each equivalence class, and establish some complexity bounds on them. Finally, we study the classification of countable complete linear orders, and give an application to classifying more general classes of countable models.

\begin{defn}
  A linear order $L$ is said to be \emph{scattered} if there does not exist an embedding (a one-to-one order-preserving map) from $\Q$ to $L$.
\end{defn}

The scattered linear orders admit a derivative or collapse operation as well as a rank function. To begin, define an equivalence relation on $L$ by $x\sim y$ if the interval between $x,y$ is finite. The $\sim$-equivalence classes are convex, so we may form a quotient ordering $L/\mathord{\sim}$. Next we define equivalence relations $\sim_\alpha$ for any ordinal $\alpha$ as follows. If $\sim_\beta$ has been defined, we let $x\sim_{\beta+1}y$ iff $[x]_\beta\sim[y]_\beta$, where $[x]_\beta,[y]_\beta$ denote the $\sim_\beta$-equivalence classes of $x,y$. If $\lambda$ is a limit ordinal and $\sim_\beta$ have been defined for all $\beta<\lambda$, we let $x\sim_\lambda y$ iff there exists $\beta<\lambda$ such that $x\sim_\beta y$.

Now a linear order $L$ is scattered if and only if there exists $\alpha$ such that $L/\mathord{\sim}_\alpha$ is a single point (see \cite[Exercise~34.18]{kechris-classical}). Thus we may define the \emph{rank} of a scattered linear order $L$ as the least $\alpha$ such that $L/\mathord{\sim}_\alpha$ is a single point. Let $S$ denote the class of countable scattered linear orders, $S_\alpha$ denote the class of countable scattered linear orders of rank $\alpha$, and $S_{\leq\alpha}$ for the class of countable scattered linear orders of rank at most $\alpha$.

\begin{prop}[Exercises 33.2, 34.18 of \cite{kechris-classical}]
  The set $S$ is a $\bPi^1_1$-complete set and the subsets $S_\alpha$ and $S_{\leq\alpha}$ are Borel sets. Moreover, the function which maps a countable scattered linear order to its rank is a $\bPi^1_1$-rank on $S$.
\end{prop}

We will write $\cong_\alpha$ to abbreviate $\cong_{S_\alpha}$. Furthermore, although $S$ is a non-Borel class of countable structures, we will write $\cong_S$ for the isomorphism relation on all countable scattered orderings.

\begin{prop}
  \label{prop:scattered-borel}
  There is no absolutely $\bDelta^1_2$ reduction from $\cong_S$ to a Borel equivalence relation. On the other hand, no Borel complete equivalence relation is Borel reducible to $\cong_S$.
\end{prop}

\begin{proof}
  For the first statement we recall from \cite{hjorth} (see remarks following Corollary~3.3) that there is no absolutely $\bDelta^1_2$ reduction from $E_{\omega_1}$ to a Borel equivalence relation, where $E_{\omega_1}$ is the isomorphism relation on countable well-orders. Moreover we have that $E_{\omega_1}$ is absolutely $\bDelta^1_2$-reducible to $\cong_S$ (see, e.g., Theorem 3.5 of \cite{clemens-coskey-potter}). Finally compositions of absolutely $\bDelta^1_2$ reductions are again absolutely $\bDelta^1_2$.
  
  For the second statement, it will follow from the results below that $F_2$ is not reducible to $\cong_S$, but we can give a more basic proof. Suppose there exists a Borel reduction $f$ from some Borel complete equivalence relation to $\cong_S$. Then by the Boundedness Theorem, the range of $f$ is contained in some $S_\alpha$. Since Borel complete equivalence relations are not Borel, it follows that $\cong_\alpha$ is not a Borel equivalence relation. But it is not difficult to see that $\cong_\alpha$ is Borel, for instance see Proposition~\ref{thm:iso-zalpha} below. This is a contradiction.
\end{proof}

The classification of scattered linear orders is closely related to the classification of $\Z$-trees. We use a modification of the definition of $\Gamma$-trees given in Section~\ref{sec:Gammatree}, which is slightly simpler and more natural in the context of orders.

\begin{defn}
  A \emph{scattered order tree} is a rooted tree together with, for each node $x$, a linear ordering on the set of immediate successors of $x$ which is isomorphic to a subordering of $\Z$.
\end{defn}

If $T$ is a scattered order tree, we will use the $<$ symbol for the order relation on the immediate successors of any node of $T$. We let $\SOT$ denote the space of scattered order trees, and $\SOT_\alpha$ the space of well-founded scattered order trees of rank $\alpha$. 

\begin{lem}
  \label{lem:iso-ztree}
  The isomorphism relation $\cong_\alpha$ on countable scattered linear orders of rank $\alpha$ is Borel bireducible with the isomorphism relation on $\SOT_{1+\alpha}$.
\end{lem}

\begin{proof}
  We first show $\cong_\alpha$ is Borel reducible to the isomorphism relation on $\SOT_{1+\alpha}$. Given a scattered linear order $L$ of rank $\alpha$ we define a scattered order tree $T$ whose nodes consist of the $\sim_\beta$ equivalence classes for $\beta\leq\alpha$, with the ordering $t\leq t'$ iff $t\subset t'$. It is not difficult to see that $L\mapsto T$ is a Borel reduction as desired.
  
  To show that the isomorphism relation on $\SOT_{1+\alpha}$ is Borel reducible to $\cong_\alpha$, we proceed by induction. For the base case, it is clear that both the isomorphism relation on $\SOT_1$ and $\cong_0$ are bireducible with $\Delta(1)$.
  
  For the inductive step, let $\alpha>0$ and assume that for all $\beta<\alpha$ there exists a Borel reduction $f_{1+\beta}$ from the isomorphism relation on $\SOT_{1+\beta}$ to $\cong_\beta$. Further assume that for all $\beta$ and all $T\in\dom(f_\beta)$ we have that $f_\beta(T)$ does not contain any $\sim$-class of size $1$. (We may do so without loss of generality, by inserting an immediate successor to each point of $f_\beta(T)$.)
  
  Given a scattered order tree $T$ of rank $1+\alpha$, let $N$ denote the order type of the children of the root of $T$, so that $N$ is a suborder of $\Z$. For each $n\in N$ let $T_n$ denote the subtree of $T$ rooted at the $n$th child of $T$, and let $\beta_n$ be the rank of $T_n$. Finally let
  \[L=\sum_{n\in N}f_{\beta_n}(T_n)+\Z+1+\Z .
  \]
  Then the mapping $T\mapsto L$ is a Borel reduction as desired. To see one can recover $T$ from $L$, note that a separator widget $\Z+1+\Z$ may be identified by its $\sim$-class of size $1$. Thus one can determine each $f_{\beta_n}(T_n)$ up to isomorphism, and then use the reductions $f_{\beta_n}$ to recover each $T_n$ and thus $T$ up to isomorphism.
\end{proof}

Note that the notion of scattered order tree carries less information than our earlier notion of $\Z$-tree, since for instance there are $\omega$ many scattered order trees of rank $2$ up to isomorphism, while there are continuum many (complexity $E_0$) $\Z$-trees of rank $2$ up to isomorphism. However, after this level the two classifications do align in complexity, with the ranks adjusted appropriately.

\begin{lem}
  \label{lem:iso-ztree-sot}
  The isomorphism relation on $\SOT_{2+\alpha}$ is Borel bireducible with $\cong^{\Z}_{\alpha}$ for $\alpha>0$.
\end{lem}

\begin{proof}
  We begin with the case when $\alpha=1$. We have already shown that $\mathord{\cong}^{\Z}_1 \mathrel{\sim_B} E_0$. Meanwhile a scattered order tree of rank 3 may be identified with a suborder of $\Z$ where each node is labeled with a natural number according to the order-type of its set of children, of which there are only countably many, and hence isomorphism of $\SOT_3$ is bireducible with $(\Delta(\omega))^{[\Z]} \mathrel{\sim_B} E_0$.
  
  For $\alpha >1$, given a scattered order tree $S$ of rank $2+\alpha$, for any node of rank at least $4$ we may first replace any of its children of rank at most $3$ by a subset of $\Z$ using the $\alpha=1$ case. We can then identify each remaining level with a subset of $\Z$ to produce a $\Z$-tree of rank $1+\alpha$ in an isomorphism-preserving way. Similarly, given a $\Z$-tree of rank $1+\alpha$, considering each ``missing'' child of a non-terminal node as a trivial tree we may replace each node of rank 1 by a scattered order tree of rank $3$ and order each level as a subset of $\Z$ to produce a scattered order tree of rank $2+\alpha$ in an isomorphism-preserving way, and yielding the desired reductions.
\end{proof}

We remark that it follows from this together with Proposition~\ref{prop:scattered-borel} that the isomorphism relation $\cong_{\SOT}$ on all scattered order trees is not Borel, but also not Borel complete.

\begin{thm}
  \label{thm:iso-zalpha}
  The isomorphism relation $\cong_{1+\alpha}$ on scattered linear orders of rank $1+\alpha$ is Borel bireducible with $Z_\alpha=\J{\Z}{\alpha}$ for $\alpha>0$.
\end{thm}

\begin{proof}
  We have $\cong_{1+\alpha}$ is bireducible with $\SOT_{2+\alpha}$ by Lemma~\ref{lem:iso-ztree}, the latter is bireducible with $\cong^{\Z}_{\alpha}$ by Lemma~\ref{lem:iso-ztree-sot}, and the latter is bireducible with $Z_\alpha$ by Proposition~\ref{prop:tree-jump}.
\end{proof}

It follows using Corollary~\ref{cor:gamma-omega-actions} that every Borel $\Z^{\omega}$-action is reducible to some $\cong_\alpha$.

It also follows using Theorem~\ref{thm:proper} that the complexity of the classification of countable scattered linear orders increases strictly with the rank. For comparison, we note that Alvir and Rossegger have shown in \cite{alvir-rossegger} that the complexity of Scott sentences of scattered linear orders also increases strictly with the rank.

We now turn to the relationship between isomorphism of scattered linear orders and assignments of scattered linear orderings to equivalence classes of countable Borel equivalence relations. We recall the following notion from \cite{kechris-amenable}.

\begin{defn}
  Let $E$ be a countable Borel equivalence relation on $X$, and let $\mathcal{K}$ be a class of $\mathcal{L}$-structures. We say that \emph{$E$ admits a Borel assignment of structures in $\mathcal{K}$ to each equivalence class} if there is an assignment $C \mapsto \mathcal{A}_C$ assigning a structure in $\mathcal{K}$ with universe $C$ to each $E$-equivalence class $C$, so that for each $k$-ary $R \in \mathcal{L}$  the relation
\[ \widetilde{R}(x,y_1,\ldots,y_k) \Leftrightarrow y_1,\ldots,y_k \in [x]_E \wedge R^{\mathcal{A}_{[x]_E}}(y_1,\ldots,y_k) \]
is Borel.
\end{defn}

A more general study of structurable equivalence relations may be found in \cite{chen-kechris}. Kechris has shown in \cite{kechris-amenable} that every countable Borel equivalence relation which admits a Borel assignment of scattered linear orders to each equivalence class is amenable, and has asked whether the converse is true. It is also a long-standing open question whether every countable amenable Borel equivalence relation is hyperfinite. We give here a partial characterization of when a countable Borel equivalence relation admits a Borel assignment of scattered linear orders to each equivalence class. 

\begin{lem}
\label{lem:scattered-assignmnet}
If a countable Borel equivalence relation $E$ admits a Borel assignment of scattered linear orders of rank $1+\alpha$ to each equivalence class, then $E \leq_B Z_{1+\alpha}$.
\end{lem}

\begin{proof}
For $\alpha=0$, Theorem 5.1 of \cite{dougherty-jackson-kechris} gives that $E$ admits an assignment of orders of rank 1 if and only if $E$ is hyperfinite, and hence reducible to $E_0 \sim_B Z_1$. For $\alpha >0$, suppose $E$ is a countable Borel equivalence relation on $X$ which admits a Borel assignment of scattered linear orders of rank $1+\alpha$ to each equivalence class. From Theorem~\ref{thm:iso-zalpha} it will suffice to show that $E \leq_B \mathord{\cong}_{2+\alpha}$. Given $x \in X$, let $<_{[x]_E}$ be the order assigned to $[x]_E$, and let $\sim_{[x]_E}$ be the finite-interval relation with respect to $<_{[x]_E}$. Let $x \mathrel{\tilde{E}} y$ iff $x \mathrel{E} y$ and $x \sim_{[x]_E} y$, so that $\tilde{E}$ is a subequivalence relation of $E$ and $<_{[x]_E}$ induces a scattered linear order of rank 1 on $[x]_{\tilde{E}}$. Hence we may fix a reduction $f$ from $\tilde{E}$ to $\cong_2$. We now assign a scattered linear order to each $x \in X$ as follows. We start with the $\tilde{E}$-equivalence classes in $[x]_E$, so that $<_{[x]_E}$ induces a scattered order of rank $\alpha$ on these classes.  For each $y \in [x]_E$ we then replace $[y]_{\tilde{E}}$ by the ordering of rank 2 given by $f(y)$ to produce an ordering of rank $2+\alpha$. Since $[x]_{\tilde{E}}$ will be encoded in this ordering via $f$, this gives the desired reduction from $E$ to $\cong_{2+\alpha}$.
\end{proof}

\begin{lem}
\label{lem:ZstarZ}
Let $E$ be a countable Borel equivalence relation. Then $E \leq_B E_0^{[\Z]}$ if and only if $E$ is $\Z \ast \Z$-orderable, i.e., $E$ admits a Borel assignment  of suborders of $\Z \ast \Z$ to each equivalence class.
\end{lem}

\begin{proof}
Suppose first that $E$ is a countable Borel equivalence relation with $E\leq_B E_0^{[\Z]}$. By Lemma~\ref{lem:z2-free-part} we have $\left(E_0\right)^{[\Z]} \leq_B \left(E_0\right)^{[\Z]}_{\text{p.i.}}$, so let $f$ be a reduction of $E$ to $\left(E_0\right)^{[\Z]}_{\text{p.i.}}$.
Recall that $E_0^\Z$ denotes the product of $\Z$-many copies of $E_0$, so $E_0^\Z$ is a countable index subequivalence relation of $E_0^{[\Z]}$ and $E_0^\Z \cong E_0^{\omega}$. Because of pairwise inequivalence, on the range of $f$ there is a Borel ordering $\prec$ of $E_0$-classes with the property that the $E_0^\Z$-classes within each $E_0^{[\Z]}$-class are isomorphic to the ordering $\Z$.
  
Let $\tilde E$ be the pullback of $E_0^\Z$ under $f$, that is, $x\mathrel{\tilde E}x'$ iff $f(x)\mathrel{E_0^\Z}f(x')$. Then $\tilde E$ is a subequivalence relation of $E$, and hence countable. Since $\tilde{E} \leq_B E_0^{\omega}$, it follows from the Sixth Dichotomy Theorem \cite[Theorem~7.1]{hjorth-kechris} that $\tilde E$ is hyperfinite. Hence there is an assignment of suborders of $\Z$ to each $\tilde{E}$-class. Moreover, letting $\prec_0$ be the pullback of $\prec$, we have that the $\tilde E$-classes within each $E$-class are isomorphic to a suborder of $\Z$. It follows that there is a Borel assignment of suborders of $\Z*\Z$ to the equivalence classes of $E$.

The converse follows from the previous lemma.
\end{proof}

Note that this is analogous to Theorem~5.1 of \cite{dougherty-jackson-kechris} which gives that a countable Borel equivalence relation is reducible to $E_0$ if and only if it admits a Borel assignment of suborders of $\Z$ to each equivalence class. We do not know whether analogous results for the forward direction hold for higher iterates of the $\Z$-jump. 
  
\begin{question}
  For a countable Borel equivalence relation $E$, does $E$ admit an assignment of scattered linear orders of rank $1+\alpha$ iff $E \leq_B \J{\Z}{\alpha}$?
\end{question}

As observed in Corollary~\ref{cor:countableZa}, if $E$ is a countable Borel equivalence relation with $E \leq_B \J{\Z}{\alpha}$ for some $\alpha < \omega_1$, then $E \leq_B \bigoplus_{n \in \omega} \J{\Z}{n}$. 
In light of Lemma~\ref{lem:scattered-assignmnet}, Corollary~\ref{cor:countableZa}, and boundedness, we have:

\begin{cor}
  If $E$ is a countable Borel equivalence relation which admits a Borel assignment of scattered linear orders to each equivalence class, then $E \leq_B \mathord{\cong}_{\omega}$.
\end{cor}

We may then ask:

\begin{question}
 If $E$ is a countable Borel equivalence relation with $E \leq_B \cong_{\omega}$, is $E$ hyperfinite?
\end{question}

Here it may be worth investigating a restricted form of the $\Gamma$-jump using finitely-supported wreath products to preserve countability of the equivalence relations, instead of the full wreath product used in the $\Gamma$-jump.

Using the above results, we can now establish the following:

\begin{cor}
\label{cor:Einfty-not-reducible}
  $E_{\infty} \not\leq_B E_0^{[\Z]}$.
\end{cor}

\begin{proof}
  Suppose towards a contradiction that $E_{\infty} \leq_B E_0^{[\Z]}$. From Lemma~\ref{lem:ZstarZ}, there would be a Borel assignment of suborders of $\Z*\Z$ to the equivalence classes of $E_{\infty}$, and therefore by \cite{kechris-amenable} we would have that $E_{\infty}$ was amenable. Since $E_\infty$ is not amenable, this contradiction completes the proof.
\end{proof}

Next we consider a class of linear orders that is very closely related to the scattered linear orders, and show that this together with our results above leads to a model-theoretic corollary.

\begin{defn}
  A linear order $L$ is said to be \emph{complete} if for every $A\subset L$, if $A$ has an upper bound then $A$ has a least upper bound.
\end{defn}

Every countable complete linear order is scattered. Indeed, if $L$ is countable and complete then it has just countably many Dedekind cuts. It follows that there is no embedding of $\Q$ into $L$, and that $L$ is scattered. On the other hand, one may verify that if $L$ is a countable scattered linear order then its Dedekind completion $\bar L$ is a countable complete linear order. In the following, we let $C$ denote the subset of $S$ consisting of the countable complete linear orders.

\begin{thm}
  The statements of Lemma~\ref{lem:iso-ztree} holds with $\cong_\alpha$ replaced with $\cong_\alpha\restriction C$. In particular for all $\alpha$ we have that the relations $Z_\alpha$, $\cong_{1+\alpha}$, and $\cong_{1+\alpha}\restriction C$ are Borel bireducible with one another.
\end{thm}

\begin{proof}
  In the proof of Lemma~\ref{lem:iso-ztree}, given a sequence $f_{\beta_n}(T_n)$ we defined a scattered order $L$ using the separator $\Z+1+\Z$. Assuming the $f_{\beta_n}(T_n)$ are complete, we can ensure that $L$ will be complete by using the separator $1+\Z+1+\Z+1$.
\end{proof}

We close this section by mentioning an application to the classification of countable models of certain theories $T$, which was pointed out to us by Ali Enayat. Let $T$ be a theory with built-in Skolem functions, and with a unary predicate $P$ and a linear ordering $<$ of $P$. Further suppose that there is a model $M$ of $T$ with a proper elementary end extension $N$, that is, $N$ is an elementary extension of $M$ and $P^M<P^N\setminus P^M$. We refer the reader to \cite{shelah} for further context.

It follows directly from Theorems~1.2 and~2.1(2) of \cite{shelah} that if $T$ is as above, then for any $\alpha$ there is a Borel reduction from $\cong_\alpha\restriction C$ to the isomorphism relation on countable models of $T$. In particular we have the following:

\begin{remark}
  If $T$ is as above, then for any $\alpha$ there is a Borel reduction from $Z_\alpha$ to the isomorphism relation on countable models of $T$.
\end{remark}

% \bibliographystyle{plain}
% \bibliography{bib}

\begin{thebibliography}{10}

\bibitem{Allison}
Shaun Allison.
\newblock Non-{A}rchimedean {TSI} {P}olish groups and their potential {B}orel
  complexity spectrum.
\newblock {\em arXiv:2010.05085}, 2020.

\bibitem{Allison-Aristotelis}
Shaun Allison and Aristotelis Panagiotopoulos.
\newblock Dynamical obstructions to classification by (co)homology and other
  {TSI}-group invariants.
\newblock {\em Trans. Amer. Math. Soc.}, 374(12):8793--8811, 2021.

\bibitem{Allison-Shani}
Shaun Allison and Assaf Shani.
\newblock Actions of tame abelian product groups.
\newblock {\em arXiv:2105.05144}, 2021.

\bibitem{alvir-rossegger}
Rachael Alvir and Dino Rossegger.
\newblock The complexity of {S}cott sentences of scattered linear orders.
\newblock {\em J. Symb. Log.}, 85(3):1079--1101, 2020.

\bibitem{becker-kechris}
Howard Becker and Alexander~S. Kechris.
\newblock {\em The descriptive set theory of {P}olish group actions}, volume
  232 of {\em London Mathematical Society Lecture Note Series}.
\newblock Cambridge University Press, Cambridge, 1996.

\bibitem{marcone}
Riccardo Camerlo, Alberto Marcone, and Luca Motto~Ros.
\newblock On isometry and isometric embeddability between ultrametric {P}olish
  spaces.
\newblock {\em Adv. Math.}, 329:1231--1284, 2018.

\bibitem{chen-kechris}
Ruiyuan Chen and Alexander~S. Kechris.
\newblock Structurable equivalence relations.
\newblock {\em Fund. Math.}, 242(2):109--185, 2018.

\bibitem{clemens-coskey-potter}
John Clemens, Samuel Coskey, and Stephanie Potter.
\newblock On the classification of vertex-transitive structures.
\newblock {\em Arch. Math. Logic}, 58(5-6):565--574, 2019.

\bibitem{clemens}
John~D. Clemens.
\newblock Relative primeness and {B}orel partition properties for equivalence
  relations.
\newblock {\em Trans. Amer. Math. Soc.}, 375(1):111--149, 2022.

\bibitem{clemens-lecomte-miller}
John~D. Clemens, Dominique Lecomte, and Benjamin~D. Miller.
\newblock Essential countability of treeable equivalence relations.
\newblock {\em Adv. Math.}, 265:1--31, 2014.

\bibitem{dougherty-jackson-kechris}
R.~Dougherty, S.~Jackson, and A.~S. Kechris.
\newblock The structure of hyperfinite {B}orel equivalence relations.
\newblock {\em Trans. Amer. Math. Soc.}, 341(1):193--225, 1994.

\bibitem{friedman-stanley}
Harvey Friedman and Lee Stanley.
\newblock A {B}orel reducibility theory for classes of countable structures.
\newblock {\em J. Symbolic Logic}, 54(3):894--914, 1989.

\bibitem{Friedman}
Harvey~M. Friedman.
\newblock Borel and {B}aire reducibility.
\newblock {\em Fund. Math.}, 164(1):61--69, 2000.

\bibitem{gao}
Su~Gao.
\newblock {\em Invariant descriptive set theory}, volume 293 of {\em Pure and
  Applied Mathematics (Boca Raton)}.
\newblock CRC Press, Boca Raton, FL, 2009.

\bibitem{harrington-kechris-louveau}
L.~A. Harrington, A.~S. Kechris, and A.~Louveau.
\newblock A {G}limm-{E}ffros dichotomy for {B}orel equivalence relations.
\newblock {\em J. Amer. Math. Soc.}, 3(4):903--928, 1990.

\bibitem{hjorth}
Greg Hjorth.
\newblock An absoluteness principle for {B}orel sets.
\newblock {\em J. Symbolic Logic}, 63(2):663--693, 1998.

\bibitem{hjorth-kechris-countable}
Greg Hjorth and Alexander~S. Kechris.
\newblock Borel equivalence relations and classifications of countable models.
\newblock {\em Ann. Pure Appl. Logic}, 82(3):221--272, 1996.

\bibitem{hjorth-kechris}
Greg Hjorth and Alexander~S. Kechris.
\newblock Recent developments in the theory of {B}orel reducibility.
\newblock {\em Fund. Math.}, 170(1-2):21--52, 2001.
\newblock Dedicated to the memory of Jerzy \L o\'{s}.

\bibitem{hjorth-kechris-louveau}
Greg Hjorth, Alexander~S. Kechris, and Alain Louveau.
\newblock Borel equivalence relations induced by actions of the symmetric
  group.
\newblock {\em Ann. Pure Appl. Logic}, 92(1):63--112, 1998.

\bibitem{jackson-kechris-louveau}
S.~Jackson, A.~S. Kechris, and A.~Louveau.
\newblock Countable {B}orel equivalence relations.
\newblock {\em J. Math. Log.}, 2(1):1--80, 2002.

\bibitem{kanovei}
Vladimir Kanovei.
\newblock {\em Borel equivalence relations}, volume~44 of {\em University
  Lecture Series}.
\newblock American Mathematical Society, Providence, RI, 2008.
\newblock Structure and classification.

\bibitem{kechris-amenable}
Alexander~S. Kechris.
\newblock Amenable equivalence relations and {T}uring degrees.
\newblock {\em J. Symbolic Logic}, 56(1):182--194, 1991.

\bibitem{kechris-classical}
Alexander~S. Kechris.
\newblock {\em Classical descriptive set theory}, volume 156 of {\em Graduate
  Texts in Mathematics}.
\newblock Springer-Verlag, New York, 1995.

\bibitem{louveau}
Alain Louveau.
\newblock On the reducibility order between {B}orel equivalence relations.
\newblock In {\em Logic, methodology and philosophy of science, {IX}
  ({U}ppsala, 1991)}, volume 134 of {\em Stud. Logic Found. Math.}, pages
  151--155. North-Holland, Amsterdam, 1994.

\bibitem{rosendal}
Christian Rosendal.
\newblock Cofinal families of {B}orel equivalence relations and quasiorders.
\newblock {\em J. Symbolic Logic}, 70(4):1325--1340, 2005.

\bibitem{shani}
Assaf Shani.
\newblock Strong ergodicity around countable products of countable equivalence
  relations.
\newblock {\em arXiv:1910.08188}, 2019.

\bibitem{shelah}
Saharon Shelah.
\newblock End extensions and numbers of countable models.
\newblock {\em J. Symbolic Logic}, 43(3):550--562, 1978.

\bibitem{solecki}
S\l{}awomir Solecki.
\newblock Equivalence relations induced by actions of {P}olish groups.
\newblock {\em Trans. Amer. Math. Soc.}, 347(12):4765--4777, 1995.

\bibitem{Ulrich-Rast-Laskowki}
Douglas Ulrich, Richard Rast, and Michael~C. Laskowski.
\newblock Borel complexity and potential canonical {S}cott sentences.
\newblock {\em Fund. Math.}, 239(2):101--147, 2017.

\bibitem{zap-pinned}
Jind\v{r}ich Zapletal.
\newblock Pinned equivalence relations.
\newblock {\em Math. Res. Lett.}, 18(3):559--564, 2011.

\end{thebibliography}

\end{document}